\documentclass{amsart}
\usepackage{amssymb,euscript,tikz,units}
\usepackage{lineno}
\usepackage[colorlinks,citecolor=blue]{hyperref}
\usepackage{verbatim}
\usepackage{colonequals}
\newcommand{\N}{\mathcal{N}}
\newcommand{\MF}{\ensuremath{\mathcal{MF}}}

\newcommand{\calT}{\mathcal{T}}
\newcommand{\T}{\calT}

\newcommand{\euS}{\EuScript{S}}
\newcommand{\R}{\mathbb{R}}

\newcommand{\Z}{\mathbb{Z}}
\renewcommand{\H}{\mathbb{H}}
\newcommand{\CC}{\mathcal{C}}
\DeclareMathOperator{\Mod}{Mod}
\newcommand{\Teich}{Teich\-m\"{u}ller~}
\DeclareMathOperator{\I}{I}

\newcommand{\Sph}[2]{\mathcal{S}_{#1}(#2)}
\newcommand{\Ball}[2]{\mathcal{B}_{#1}(#2)}
\newcommand{\Ann}[3]{\mathcal{A}^{#3}_{#1}(#2)}

\newcommand{\Nbhd}{\mathop{\rm Nbhd}\nolimits}

\DeclareMathOperator{\diam}{diam}

\newcommand{\qd}{\mathcal{Q}}
\newcommand{\QD}{\qd}
\newcommand{\diskbund}{\QD^{\leq 1}}
\newcommand{\floor}[1]{\left\lfloor #1 \right\rfloor}
\newcommand{\ceil}[1]{\left\lceil #1 \right\rceil}
\newcommand{\thresh}[1]{\left[ #1 \right]}
\newcommand{\plog}{\log_+\negthinspace}
\newcommand{\abs}[1]{\lvert #1 \rvert}
\newcommand{\norm}[1]{\lVert #1 \rVert}
\newcommand{\D}{{\sf D}}
\newcommand{\M}{{\sf M}}
\newcommand{\B}{{\sf B}}
\newcommand{\ccc}{{\sf c}}
\renewcommand{\P}{{\sf P}}
\renewcommand{\L}{{\sf L}}

\newcommand{\emul}{\overset{.}{\asymp}}
\newcommand{\gmul}{\overset{.}{\succ}}
\newcommand{\lmul}{\overset{.}{\prec}}

\newcommand{\eboth}{\asymp}
\newcommand{\gboth}{\succ}
\newcommand{\lboth}{\prec}

\newcommand{\vismeas}{\nu_{x}}
\newcommand{\hausmeas}{\eta}
\newcommand{\m}{\mathbf m}
\newcommand{\n}{\mathbf n}
\newcommand{\thumeas}{\mu_{\rm \scriptscriptstyle TH}}
\newcommand{\busemeas}{\mu_{\rm \scriptscriptstyle B}}
\newcommand{\htmeas}{\mu_{\rm \scriptscriptstyle HT}}
\newcommand{\muhol}{\mu_{\rm \scriptscriptstyle hol}}
\newcommand{\sympmeas}{\mu_{\rm \scriptscriptstyle sp}}
\newcommand{\Mah}{M}
\DeclareMathOperator{\vis}{\rm Vis}
\newcommand{\thickfrac}{{\sf Thk}^\%}
\DeclareMathOperator{\len}{length}

\newcommand{\sdot}{\! \cdot \!}

\theoremstyle{plain}
\newtheorem{theorem}{Theorem}[section]
\newtheorem{corollary}[theorem]{Corollary}
\newtheorem{proposition}[theorem]{Proposition}
\newtheorem{lemma}[theorem]{Lemma}
\newtheorem{claim}[theorem]{Claim}

\newtheorem{mainthrm}{Theorem} 
\newtheorem*{theorem:triangles}{Theorem~\ref{thm:partiallyThickIntervalsAreClose}}
\newtheorem*{theorem:balls}{Theorem~\ref{thm:balls}}
\newtheorem*{theorem:spheres}{Theorem~\ref{thm:spheres}}
\newtheorem*{theorem:expthin}{Theorem~\ref{thm:expthin}}

\theoremstyle{definition}
\newtheorem{definition}[theorem]{Definition}

\theoremstyle{remark}
\newtheorem{rem}[theorem]{Remark}
\newtheorem*{rem*}{Remark}
\newtheorem*{rems*}{Remarks}

\makeatletter
\let\c@equation\c@theorem
\makeatother
\numberwithin{equation}{section}

\pgfdeclarelayer{background}
\pgfsetlayers{background,main}

\begin{document}

\title{Statistical hyperbolicity in Teichm\"uller space}

\author[Dowdall]{Spencer Dowdall}
\address{
Spencer Dowdall\\
Mathematics\\
University of Illinois at Urbana-Champaign\\
\phantom{Urbana}
\indent Urbana, IL 61801}
\email{dowdall@math.uiuc.edu}
\author[Duchin]{Moon Duchin}
\address{
Moon Duchin\\
Mathematics\\
Tufts University\\
Medford, MA 02155}
\email{moon.duchin@tufts.edu}
\author[Masur]{Howard Masur}
\address{
Howard Masur\\
Mathematics\\
University of Chicago\\
Chicago, IL 60637}
\email{masur@math.uchicago.edu}

\date{\today}

\thanks{The first author was partially supported by the NSF postdoctoral research fellowship DMS-1204814. The second author was partially supported by NSF DMS-0906086. 
The third author was partially supported by NSF DMS-0905907.}

\begin{abstract}
In this paper we explore the idea that Teichm\"uller space is hyperbolic ``on average.'' Our approach focuses on studying the geometry of geodesics which spend a definite proportion of time in some thick part of Teichm\"uller space. We consider several different measures on Teichm\"uller space and find that this behavior for geodesics is indeed typical. With respect to each of these measures, we  show that the average distance between points in a ball of radius $r$ is asymptotic to $2r$, which is as large as possible. Our techniques also lead to a statement quantifying the expected thinness of random triangles in Teichm\"uller space, showing that ``most triangles are mostly thin.''
\end{abstract}

\maketitle
%%%%%%%%

\section{Introduction}
Let $S$ be a closed surface of genus $g>1$.  In this paper we continue the study of metric properties of
Teichm\"uller space 
$\T(S)$, which is the parameter space for several types of geometric structures on $S$.
Equipped with the Teichm\"uller metric $d_\T$, it
is a complete metric space homeomorphic to $\R^{6g-6}$. It is not $\delta$--hyperbolic \cite{MW}, and several kinds of obstructions to hyperbolicity are known:  for instance, 
pairs of geodesic rays through the same point may  fellow-travel arbitrarily far apart \cite{M75}, 
and there are large ``thin parts'' of the space which, up to bounded additive error, are isometric to product spaces 
equipped with sup metrics  (and therefore rule out hyperbolicity in the space as a whole) \cite{Min-PR}.
However, these exceptions to negative curvature seem to come from rare occurrences, while a long list of properties associated with hyperbolicity do hold globally or in specialized situations. Thus one may expect such properties to hold generically or ``on average.'' 
This paper aims to show that this is indeed the case.

Our motivating goal is to understand the generic geometry of Teichm\"uller space, particularly with regard to negative-curvature phenomena. For example, geodesics that stay in the thick part of $\T(S)$ are well understood and exhibit many properties characteristic of hyperbolicity. Geodesics lying completely in the thin part are also well understood, and exhibit no negative-curvature characteristics. However, much more typical is for a geodesic to spend time in both the thick and thin parts of 
$\T(S)$---indeed, a generic geodesic ray will switch between these parts infinitely often. In this paper we develop tools to study 
these types of geodesics, and we discover that certain negative-curvature phenomena do hold in this setting. 
For example, we obtain the following variant of the thin triangle property.

\begin{mainthrm}\label{thm:partiallyThickIntervalsAreClose}
For any $\epsilon>0$ and $0<\theta \leq 1$, there exist constants $C,L$ such that 
if $I\subset[x,y]\subset \T(S)$ is a geodesic subinterval of length at least $L$ 
and at least proportion $\theta$ of $I$  is $\epsilon$--thick,
then for all $z\in\T(S)$, we have 
\[ I \cap \Nbhd_C([x,z]\cup[y,z]) \neq \emptyset. \]
\end{mainthrm}

This is a generalization of a result of Rafi, Theorem~\ref{thm:RafiThinTriangles} below, which gives the same conclusion under the stronger hypothesis that the entire interval $I$ is thick. 
As a consequence (Corollary~\ref{cor:halfThinTriangles}), we can assert for instance that if the three sides of a triangle each spend more than half their time in the $\epsilon$--thick part, then half of each side is within a uniformly bounded distance of the union of the other two sides.

Since our goal is to study the geometry of generic rays, Theorem~\ref{thm:partiallyThickIntervalsAreClose} motivates a consideration of whether randomly sampled geodesics are, with high probability, likely to spend a definite fraction of time in a given thick part. To address this, we investigate a number of a priori different measures on Teichm\"uller space, which are natural from various points of view, and show that having a definite proportion in the thick part is in fact typical for all of them.

For instance, as a metric space $\T(S)$ carries a $(6g-6)$--dimensional Hausdorff measure
$\hausmeas$. This allows us to fix $x$ and choose $y\in \Ball rx$ at random from the $r$--ball as a way of sampling geodesics 
$[x,y]$.
Other measures come from the Finsler structure
(Busemann measure $\busemeas$ and Holmes--Thompson measure $\htmeas$), from the holonomy coordinates on the cotangent bundle (holonomy, or Masur--Veech, measure $\m$), and from the symplectic structure. In \S\ref{s:measures}, we find that all of these measures are mutually absolutely continuous and in fact are related by explicit inequalities.
Further interesting measures are provided by the identification of the metric $r$--sphere $\Sph rx$ with the unit sphere 
$\qd^1(x)$ in the vector space of 
quadratic differentials on $x$ via the Teichm\"uller map.  
The latter has various natural measures, and corresponding measures on $\Sph rx$ will be called {\em visual measures};
we will pay special attention to two standard visual measures, denoted $\vis_r(\vismeas)$ and $\vis_r(s_x)$. 
We will also use the term visual measures for the induced measures on $\T(S)$, denoted 
$\vis(\vismeas)$ and $\vis(s_x)$, obtained by integrating radially.

As one application of our statistical approach, we compute 
a statistic built by combining a metric and a measure to quantify how fast a space ``spreads out.''
Suppose we are given a  family of
probability measures $\mu_r$ on the spheres $\Sph rx$ of a metric space $(X,d)$. 
Then let $E(X)=E(X,x,d,\{\mu_r\})$ be  the average normalized distance between points on large spheres:
$$E(X):=\lim_{r\to\infty} \frac{1}{r} \int_{\Sph rx \times \Sph rx} d(y,z) \ \ d\mu_r(y) d\mu_r(z),$$
if the limit exists.   This creates a numerical index varying from $0$ (least spread out) to $2$ (most spread out).
It is shown in \cite{DLM2} that non-elementary hyperbolic groups all have 
$E(G,S)=2$ for any finite generating set $S$; this is also the case in the hyperbolic space $\H^n$ of 
any dimension endowed with the natural measure on spheres.
By contrast, it is shown that $E(\R^n)<\sqrt 2$ for all $n$, and that 
$E(\Z^n,S)<2$ for all $n$ and $S$, with nontrivial dependence on $S$.  (See \cite{DLM2} for more examples.)
Motivated by these findings, we may regard a measured metric space with $E=2$ as being ``statistically hyperbolic."

We note that finding that $E=2$ for hyperbolic groups makes use of homogeneity.
In contrast, it is easy to build (highly non-regular) locally finite trees, equipped with counting measure on spheres, for which $E$ obtains any value from $0$ to $2$; see \cite[p.4]{DLM2}. Thus neither $\delta$--hyperbolicity nor exponential growth is sufficient to ensure $E=2$. 
Indeed, since the measures are normalized, the growth rate of the space has no direct effect on $E$.
As an illustration, note that  
the Euclidean plane could be endowed with a visual measure, constructed just like the ones we study below in 
\S\ref{vismeas}, which would give it exponential volume growth while leaving $E=4/\pi$ unchanged.  On the other hand,
other measures on $\R^2$ would give different values of $E$;
the statistic is quite sensitive to the choice of measure.

The following theorem concerns the average distance between points in the ball $\Ball{r}{x}$ of radius $r$ centered at $x$. 
We show that this average distance is asymptotic to 
$2r$, which, in light of the triangle inequality, is the maximum possible distance.

\begin{mainthrm}\label{thm:balls}
Let $\mu$ denote the Hausdorff measure $\hausmeas$, holonomy measure $\m$, or either standard visual measure $\vis(\vismeas)$ or $\vis(s_x)$.
Then for every point $x\in \T(S)$,
\[\lim_{r\to\infty} \frac 1r \frac{1}{\mu (\Ball rx)^2}\int_{\Ball rx \times \Ball rx} d_{\T}(y,z) \ \ 
d\mu(y) d\mu(z)=2.\]
\end{mainthrm}

Of course, by the remarks above, this also holds for all the other measures discussed in the paper.
Indeed, we will work with  properties  of  measures on $\T(S)$ that suffice to guarantee this conclusion:  
a {\em thickness property} (P1) defined in \S\ref{s:thick-stat} guaranteeing that typical rays spend
a definite proportion of their time in the thick part, and 
a {\em separation property} (P2) defined in \S\ref{s:fellow-trav-stat} asserting that typical pairs of rays will exceed any definite amount of separation.  In some places we use a stronger separation property (P3) which is 
a quantified version with an exponential bound.

With respect to the standard visual measures, the same methods yield:
\begin{mainthrm}
\label{thm:spheres}
For every point $x\in \T(S)$ and  either family $\{\mu_r\}$ of standard visual measures $\mu_r =\vis_r(\vismeas)$ or $\vis_r(s_x)$ on the spheres $\Sph{r}{x}$, we have
\[E(\T(S),x,d_\T,\{\mu_r\}) = 2.\]
\end{mainthrm}

As a second application of our approach, we promote Theorem~\ref{thm:partiallyThickIntervalsAreClose} to a quantitative statement about the expected thinness of typical triangles. 
This is expressed in the following theorem, which shows that ``most triangles are mostly thin.''
For a fixed $\delta$, let  $0\leq \Theta_\delta(\triangle)\leq 1$ denote the proportion of the perimeter of a geodesic triangle $\triangle$ that lies within $\delta$ of the other two sides.
Then let $\Theta_\delta(X) = \Theta_\delta(X, x, d, \mu)$ be the limiting average of this value: 
$$\Theta_\delta(X) := \liminf_{r\to\infty} \frac{1}{\mu (\Ball rx)^2}\int_{\Ball rx \times \Ball rx} \Theta_\delta(\triangle(x,y,z)) \ d\mu(y)d\mu(z).$$

\begin{mainthrm}
\label{T:expected_thinness}\label{thm:expthin}
Let $\mu$ denote either the Hausdorff measure $\hausmeas$ or the holonomy measure $\m$. Then for all $x\in\T(S)$ and $\sigma > 0$ there exists  $\delta>0$ such that
$$\Theta_\delta(\T(S),x,d_\T,\mu) \geq 1-\sigma.$$
\end{mainthrm}

In other words, the proportion of a triangle's perimeter that is close to the other two sides can be made arbitrarily close to $1$ in expectation. 
By contrast, note that $\Theta_\delta(\R^n)=0$ for all $\delta$, whereas $\delta$--hyperbolic spaces $X$ automatically satisfy $\Theta_\delta(X)=1$ by definition.

We sketch here the main ideas in the proofs of the theorems. Theorem~\ref{thm:partiallyThickIntervalsAreClose} is put together with distance estimates coming from subsurface projections, using reverse triangle inequalities (following Masur--Minsky and Rafi), and antichain bounds (Rafi and Schleimer).  A crucial ingredient is to show that geodesics spending a definite proportion of time in the thick part have shadows that make definite progress in the curve complex (Theorem~\ref{prop:def-prog}).

The idea for Theorem~\ref{thm:balls} is that the  separation  property (P2) ensures that  most  pairs of geodesics will have stopped fellow-traveling in the Teichm\"uller metric by a threshold time.  Then one would hope that, as in a hyperbolic space, the geodesic joining their endpoints would follow the first geodesic back to approximately where they  separate before  following the other so that its length is roughly the sum of the lengths of the two geodesics, as on the left in Figure~\ref{schematic}.  
\begin{figure}[ht] 
\begin{tikzpicture}[scale=.4,inner sep=0]

\begin{scope}[line width=1pt]
\draw [->] (0,0) .. controls (-1.5,2) .. (-3,6) node [pos=.92] (x1) {} ;
\draw [->] (0,0) .. controls (1.5,2) ..  (3,6) node [pos=.92] (x2) {} ;
\draw [red] (x1) .. controls (0,-.5) .. (x2)  ;
\filldraw (0,0) circle (0.05) node [below=.2cm] {$x$};
\filldraw (x1) circle (0.05) node [ left=.2cm] {$y$};
\filldraw (x2) circle (0.05) node [right=.2cm] {$z$};
\end{scope}

\begin{scope}[line width=1pt,xshift=10cm]
\draw [->] (0,0) .. controls (-1.5,2) .. (-3,6) node [pos=.92] (x1) {} node [pos=.65] (x11) {} ;
\draw [->] (0,0) .. controls (1.5,2) ..  (3,6) node [pos=.92] (x2) {} node [pos=.6] (x22) {};
\filldraw (0,0) circle (0.05) node [below=.2cm] {$x$};
\filldraw (x1) circle (0.05) node [ left=.2cm] {$y$};
\filldraw (x2) circle (0.05) node [right=.2cm] {$z$};
\filldraw (x11) circle (0.05) node [ left=.2cm] {$y'$};
\filldraw (x22) circle (0.05) node [right=.2cm] {$z'$};
\draw [red] (x1) .. controls (-2.2,4).. (-1.5,3) .. controls (0,2.7) .. (1.5,2.8) .. controls (2.2,4)..(x2) ;
\end{scope}
\end{tikzpicture}
\caption{We will show that the geodesic between points on generic rays ``dips back'' near the basepoint. While Minsky's product regions theorem says that the connecting geodesic can instead take a ``shortcut'' when $[x,y']$ and $[x,z']$ go through thin parts corresponding to disjoint subsurfaces,  we show this effect is rare.\label{schematic}}
\end{figure}
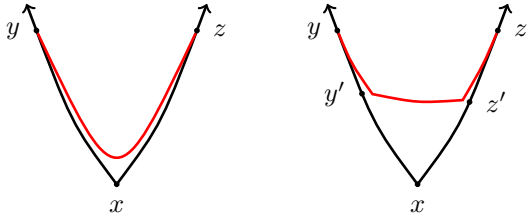

One obstruction to this hyperbolic-like behavior is that the pair of geodesics can enter thin parts corresponding to disjoint subsurfaces, in which case Minsky's product region theorem \cite{Min-PR} allows the length of the third side to be smaller than the sum, as on the right in Figure~\ref{schematic}.  
The thickness property (P1) and Theorem~\ref{thm:partiallyThickIntervalsAreClose} together rule out this shortcut behavior.
Theorem~\ref{thm:spheres} follows from this and exponential growth of the metric.

Theorem~\ref{T:expected_thinness} uses Theorem~\ref{thm:partiallyThickIntervalsAreClose} and a strengthening of the thickness property (P1) to find  thick points in various locations around a typical triangle. 
The proof then concludes by applying Rafi's fellow-traveling theorem \cite{Ra} to deduce that most of each side lies close to other sides.

In establishing the needed thickness and separation properties for the above results, we  use a variety of recently developed tools such as volume asymptotics in Teichm\"uller space (Athreya--Bufetov--Eskin--Mirzakhani \cite{ABEM}) and the random walk model for discretized Teichm\"uller geodesics  (Eskin--Mirzakhani \cite{EM}). A detailed treatment of the latter is included in Appendix~\ref{sec:thickness_for_random_walks}. 
We also make use of a simplified version of Rafi's distance formula \cite{R07}, which is derived in Appendix~\ref{repackaged}.

\subsection{Acknowledgments}
We would like to thank Benson Farb,   Curtis McMullen, and  especially Alex Eskin and Kasra Rafi for numerous helpful comments and explanations.  In particular, we are indebted to Kasra Rafi for encouraging us to investigate expected thinness (Theorem~\ref{T:expected_thinness}) and for suggesting the idea of the proof for Theorem~\ref{thm:interval}.
We are also grateful for the suggestions of the anonymous referees, which led to major improvements in the paper.

%%%%%%%%
\section{Background}

\subsection{Teichm\"uller space and quadratic differentials}

Recall that Teichm\"uller space $\T(S)$ is the space of marked  Riemann surfaces $X$ that are homeomorphic to the topological surface $S$. More precisely, it consists of pairs $(X,f)$, where  $f\colon S\to X$ is a  homeomorphism, up to the equivalence relation that $(X_1,f_1)\sim (X_2,f_2)$ when there exists a conformal map 
$F\colon X_1\to X_2$ such that  $F\circ f_1$ is isotopic to $f_2$. 
Alternately, we may define  $\T(S)$ as the space of marked hyperbolic surfaces $(\rho,f)$; namely, 
the markings are maps $f\colon S\to \rho$ with
$(\rho_1,f_1)\sim (\rho_2,f_2)$ when there exists an isometry $F\colon\rho_1\to \rho_2$ such that $F\circ f_1$ is isotopic to $f_2$.

The space $\T(S)$ is homeomorphic to the ball $\R^{6g-6}$, and from now on we will use $h=6g-6$ to designate this dimension.
In this paper, we will typically denote a point of $\T(S)$ by $x$, regarding it either as a Riemann surface or a hyperbolic surface, and 
suppressing the marking $f$.

Using the first definition of $\T(S)$, 
 the Teichm\"uller distance is given by
\[d_\T((X_1,f_1), (X_2,f_2)):=\inf_{F\sim f_2\circ f_1^{-1}}\frac{1}{2} \log K(F),\]
where the minimum is taken over all quasiconformal maps $F$ and  $K(F)$ is the 
maximal dilatation of $F$. Equipped with this metric, Teichm\"uller space becomes a unique geodesic metric space. For $x,y\in \T(S)$, the Teichm\"uller geodesic segment joining $x$ to $y$ will usually be denoted $[x,y]$.  
We will also write $y_t$ for the time--$t$ point on the ray based at $x$ and going through $y$.

A {\em quadratic differential} on a Riemann surface $X$ is a holomorphic $2$--tensor
$q=\phi(z)dz^2$ on $X$.
The space of all quadratic differentials on all Riemann surfaces homeomorphic to $S$ is denoted $\QD(S)$. A point of $\QD(S)$ will be denoted 
$q$, with the underlying 
complex structure implicit in the notation.  The real dimension of $\QD(S)$ is $12g-12=2h$.  
Reading off the Riemann surface, we obtain a projection to the Teichm\"uller space  $\pi\colon \QD(S) \to \T(S)$. Under this projection, $\QD(S)$ forms vector bundle over $\T(S)$ which is canonically identified with the cotangent bundle of $\T(S)$.
Each fiber $\QD(X)$ is equipped with a norm given by the total area of $q$; namely $\|q\|=\int_X |\phi(z)dz^2|$. 
Recall that $d_\T$  is not a Riemannian metric on $\T(S)$, but rather a Finsler metric; it comes from 
dualizing the norm on $\QD$ to give a norm on each tangent space of $\T(S)$ that is not induced by any inner product.

It is the famous theorem of Teichm\"uller that the infimum in the definition of $d_\T$ 
is realized uniquely by a {\em Teichm\"uller map} 
from $X_1$ to $X_2$.  
A Teichm\"uller map is determined by an initial quadratic differential $q=\phi(z)dz^2$ on 
$X_1$ and the number $K$. The Teichm\"uller map expands along the 
horizontal trajectories of $q$ by a factor of $K^{1/2}$ and contracts along the vertical trajectories by the same factor to obtain 
a terminal quadratic differential $q'$ on the image surface $X_2$. 
If we fix $q$ and let $K=e^{2t}$ vary over $t\in [0,\infty)$ we get a Teichm\"uller geodesic ray. 
We will denote by $\nu_+$ the horizontal foliation of $q$ and by $\nu_-$ the vertical foliation.

Recall that the {\em mapping class group} of $S$, defined by $$\Mod(S):={\rm Diff}^+(S)/{\rm Diff}_0(S),$$
is the discrete group of orientation-preserving diffeomorphisms of $S$, up to isotopy. This group acts isometrically on $\T(S)$ by changing the marking: $\phi\cdot(X,f) = (X,f\circ\phi^{-1})$. In fact, by a result of Royden \cite{Royden71}, $\Mod(S)$ is the full group of (orientation-preserving) isometries of $(T(S),d_\T)$.

%%%%%%%%%%%
\subsection{Curve complex}
\label{sec:curve_graph}
When we speak of a {\em curve} on $S$, this will mean an isotopy class of essential simple closed curves. Given $x\in \T(S)$, the \emph{length} $l_x(\alpha)$ of a curve $\alpha$ is the  length of the geodesic in the isotopy class in the hyperbolic metric $x$.

We recall the definition of the  {\em curve complex} (or curve graph)  
$\CC(S)$ of $S$. The vertices of $\CC(S)$ are the curves on $S$.  
Two vertices are joined by an edge if the corresponding curves can be realized disjointly.
Assigning  edges to have length $1$ we have a metric graph. Properly speaking, $\CC(S)$ is the flag complex associated
to this curve graph, but since we are working coarsely, we can identify $\CC(S)$ with the graph.

It is known that the curve graph  is hyperbolic \cite{MM}.
That is, there exists a constant $\delta>0$ such that every geodesic triangle in $\CC(S)$ is $\delta$--thin:  
each side of the triangle is contained in the union of the $\delta$--neighborhoods of the other two sides. 
Furthermore, in any $\delta$--hyperbolic metric space and for any quasi-isometry constants $(K,C)$, there exists a constant $\tau$, depending only on $\delta,K,C$, such that any two 
$(K,C)$--quasi-geodesic segments with the same endpoints remain within $\tau$ of each other. Since actual geodesics are $(1,0)$--quasi-geodesics, this implies that every $(K,C)$--quasi-geodesic triangle is $(\delta+2\tau)$--thin.

%%%%
\subsection{Thick parts and subsurface projections}
\label{sec:subsurface_projections}

For any given $\epsilon$, we say a curve is $\epsilon$--short if its hyperbolic length is less than $\epsilon$.
Then define the {\em $\epsilon$--thick part} of Teichm\"uller space to be the subset 
$\T_{\epsilon}\subset\T(S)$ corresponding to those hyperbolic surfaces on which no curve is $\epsilon$--short.
Its complement is called the {\em $\epsilon$--thin part} or, when $\epsilon$ is understood, simply the thin part.

For each $x\in \T(S)$  there is associated a  {\em Bers marking} $\mu_x$. To construct $\mu_x$, greedily choose a shortest {\em pants decomposition} of the surface (a collection of $3g-3$ disjoint simple geodesics). Then for each pants curve $\beta$, choose a shortest  geodesic  crossing $\beta$ minimally (either once or twice depending on the topology) that is disjoint from all other pants curves. The total collection of $6g-6$ curves is called a Bers marking and is defined up to finitely many choices. Notice that the curves comprising $\mu_x$ form a diameter--$2$ subset of $\CC(S)$.

Recall that there exists a universal \emph{Margulis constant} such that any two curves with hyperbolic length (on any surface $x\in \T(S)$) less than this value are disjoint. When discussing the $\epsilon$--thick part $\T_\epsilon$, we always assume $\epsilon$ is less than the Margulis constant. In particular, this ensures that for $x\in \T(S)\setminus\T_\epsilon$, the Bers marking $\mu_x$ contains every curve $\alpha$ with $l_x(\alpha)\leq \epsilon$.

Throughout, a \emph{proper subsurface of $S$} will mean a compact, properly embedded subsurface $V\subset S$ which is not 
equal to $S$ and for which the induced map on fundamental groups is injective. Subsurfaces which are isotopic to each other 
will not be considered distinct. The proper subsurfaces of $S$ fall into two categories, annuli and non-annuli, which behave 
somewhat differently. Nevertheless, we will strive to develop intuitive notation under which these two possibilities may be dealt 
with on equal footing.

Every proper subsurface $V$ has a nonempty boundary $\partial V$ consisting of a disjoint union of  
curves on $S$. We say that two subsurfaces $V$ and $W$ \emph{transversely intersect}, denoted $V\pitchfork W$, 
if they are neither (isotopically) disjoint nor nested. In this case, $\partial V$ necessarily intersects $W$, and $\partial W$ 
intersects $V$.  

Consider a non-annular subsurface $V$, possibly equal to $S$. 
The \emph{subsurface projection} $\pi_V(\beta)$ of a simple closed curve $\beta\subset S$ to $V$ is defined as follows: 
Realize $\beta$ and $\partial V$ as geodesics (in any hyperbolic metric on $S$). If $\beta\subset V$, then $\pi_V(\beta)$ is defined to be $\beta$. 
If $\beta$ is disjoint from $V$, then $\pi_V(\beta)$ is undefined. Otherwise, $\beta\cap V$ is a disjoint union of finitely many homotopy classes 
of arcs with endpoints on $\partial V$, and we obtain $\pi_V(\beta)$ by choosing any arc and performing a surgery along $\partial V$ to create 
a simple closed curve contained in $V$. The subsurface projection of a point $x\in\T(S)$ is then defined to be the collection
\[\pi_V(x) := \{\pi_V(\beta)\}_{\beta\in\mu_x}\]
 of curves obtained by varying $\beta$ in the Bers marking at $x$. This is a non-empty subset of the curve complex $\CC(V)$ with 
 uniformly bounded diameter. 

\begin{definition}[Non-annular projection distance]
For a non-annular subsurface $V\subseteq S$, the \emph{projection distance in $V$} of a pair of points $x,y\in\T(S)$ is defined to be
\[d_V(x,y):=\text{diam}_{\CC(V)}(\pi_V(x)\cup\pi_V(y)).\]
In particular, $d_S(x,y)$ denotes the curve complex distance. When convenient, we will also denote this distance by $d_{\CC(V)}:=d_V$.
\end{definition}

For an annular subsurface $A\subset S$ with \emph{core curve} $\alpha = \partial A$, there are two kinds of projection distances: one that measures twisting about $\alpha$ and is analogous to the definition above, and a second which also incorporates the length of $\alpha$. Any simple closed curve $\beta$ that crosses $\alpha$ may be realized by a geodesic and then lifted to a geodesic $\tilde{\beta}$ in the annular cover $\tilde{A}$, that is, the quotient of $\H^2$ by the deck transformation corresponding to $\alpha$, with the Gromov compactification.    For a pair $\beta, \gamma$ of such curves, we may then consider the intersection number $\textrm{i}(\tilde{\beta},\tilde{\gamma})$ in $\tilde{A}$. The \emph{twisting distance} in $A$ of a pair of points $x,y\in \T(S)$ is then defined as
\[d_{\CC(A)}(x,y):= \sup_{\beta\in \mu_x, \gamma\in \mu_y} \textrm{i}_{\tilde{A}}(\tilde{\beta},\tilde{\gamma}).\]
We additionally define a hyperbolic projection distance as follows.

\begin{definition}[Annular projection distance]
\label{def:annular_dist}
For an annular subsurface $A\subset S$ with core curve $\alpha = \partial A$, we let $\H_\alpha$ denote a copy of the standard horoball $\{\text{Im}(z)\geq 1\}\subset \H^2$. 
Given $x,y\in \T(S)$, we consider the points $(0,1/l_x(\alpha))$ and $(d_{\CC(A)}(x,y),1/l_y(\alpha))\in \H^2$ and denote their closest point projections to the horoball $\H_\alpha$ by
\[\pi_\alpha(x) = \left(0,\max\left\{1,\frac{1}{l_x(\alpha)}\right\}\right),\quad
\pi_\alpha(y) = \left(d_{\CC(A)}(x,y),\max\left\{1,\frac{1}{l_y(\alpha)}\right\}\right).\]
The \emph{projection distance} in $A$ (or \emph{hyperbolic distance} $d_{\H_\alpha}$) between $x$ and $y$ is then defined to be
\[d_A(x,y):= d_{\H^2}\left(\pi_\alpha(x),\pi_\alpha(y)\right).\]
\end{definition}

\subsection{Notation}
\label{sec:notation}

Following Rafi \cite[\S2.4]{R07}, we fix a parameter $\epsilon_0 > 0$ for the entirety of this paper which is smaller than the Margulis constant and small enough for a few other fundamental results to hold 
(Minsky's product regions theorem and Rafi's distance estimates described in the following section).
Note that the definition of $\epsilon_0$ depends only on the topology of the surface $S$, and we therefore view $\epsilon_0$ as a global constant.

Our analysis involves many inequalities that have controlled multiplicative and additive error. To streamline the the presentation, we will often avoid explicitly writing the constants involved and will instead rely on the following notation: For real-valued expressions $A$ and $B$, we use the notation
\[A\lmul B\]
to mean that there exists a universal constant $c\geq 1$, depending only on the topology of the surface $S$, such that $A\leq c B$. We will use $A\emul B$ to mean that $A\lmul B$ and $A\gmul B$ both hold. (The dot in the symbols indicates that the error is only multiplicative.) When allowing for multiplicative and additive error we will instead use symbols $\lboth$, $\gboth$, and $\eboth$. Thus $A\eboth B$ means that there exists a universal constant $c\geq 1$ so that $A\leq cB+c$ and $B\leq cA+c$.

When the implied constant depends on additional parameters we will list these as subscripts of the binary relation. For example, $A\lmul_{\epsilon,\theta} B$ means that there exists a constant $c$ depending only on $\epsilon$, $\theta$, and the topology of $S$ such that $A\leq cB$.

\subsection{Distance formula}\label{s:dist-form}
The following distance formula due to Rafi relates the \Teich distance between two points $x$ and $y$ to the combinatorics of the corresponding Bers markings $\mu_x$ and $\mu_y$. Recall the global constant $\epsilon_0 > 0$ introduced in \S\ref{sec:notation} above.

\begin{theorem}[Distance formula, Rafi \cite{R07}]\label{thm:RafiDistFormula}
Given any sufficiently large threshold $M_0$, for all $x,y\in \T(S)$ we have
\begin{align*}
 d_\T(x, y )\; &\asymp_{M_0}
\; d_S(x,y) 
+ \sum_{V} \thresh{d_V(x,y)}_{M_0}
+ \max_{\alpha\in\Gamma_{xy}}d_{\H_\alpha}(x,y)\\
&\ + \!\!\!\sum_{A \,:\, \partial A \not \in \Gamma_{xy}}\!\!\!\!\plog\thresh{d_{\CC(A)}(x,y)}_{M_0}
+\max_{\alpha \in \Gamma_x}\,\plog\left(\frac{1}{l_{x}(\alpha)}\right)
 +\max_{\alpha \in \Gamma_y}\,\plog\left(\frac{1}{l_{y}(\alpha)}\right),
\end{align*}
where the first sum is over all non-annular proper subsurfaces $V\subsetneq S$, where $\Gamma_{xy}$ is the set of $\epsilon_0$--short curves in both $x$ and $y$, $\Gamma_x$ is the set of curves that are $\epsilon_0$--short 
in $x$ but not in $y$, and $\Gamma_y$ is defined similarly. Here and throughout, $\plog$ is a modified logarithm so that $\plog a = 0$ for $a\in[0,1]$; and $\thresh{\cdot}_{M_0}$ is a threshold function for which $\thresh{N}_{M_0} := N$ when $N\geq M_0$ and $\thresh{N}_{M_0} := 0$ otherwise.
\end{theorem}

By instead making all annular measurements with the hyperbolic distance on $\H_\alpha$ we will obtain a particularly simple restatement of this formula. 

\begin{proposition}[Repackaged distance formula]
\label{prop:repackage}
Given any sufficiently large threshold $M_0$, for all $x,y\in \T(S)$ we have:
\begin{equation}\label{dist}
  d_\T(x,y)\ 
\asymp_{M_0}\ 
    d_S(x,y) 
+ \sum_Y \thresh{d_Y(x,y)}_{{M_0}}
\end{equation}
Here, the sum is over all (annular and non-annular) proper subsurfaces. 
\end{proposition}

\begin{rem}
Our definition of $d_{\H_\alpha}=d_A$ is technically different than that used by Rafi in \cite{R07}; however, the two definitions agree up to bounded additive error.
\end{rem}

 In calling it ``repackaged" we mean to say that the content of (\ref{dist}) is essentially contained in
Rafi \cite{R07}.  We include a detailed proof here in Appendix~\ref{repackaged}.

%%%%
\subsection{Thin intervals}
\label{sec:thinness}
We will use some results from Rafi's work  combinatorializing the Teichm\"uller metric. Specifically, Corollary 3.4 and Proposition 3.7 of \cite{R07} show that for every Teichm\"uller geodesic and every proper subsurface $V$, there is a (possibly empty) interval along the geodesic where $\partial V$  is short. Outside of this interval, the  projections $d_V$ move by at most a bounded amount.
In the form that we will use below: for each positive $\epsilon \leq \epsilon_0$
there are positive constants $\M_\epsilon$ and $\epsilon'\leq\epsilon$
such that for any pair of points $x,y\in \T(S)$  there is a possibly empty (and not uniquely defined) connected interval $\I_V^\epsilon$ 
along the geodesic segment $[x,y]$ such that 
\begin{itemize}
\item for $a\in \I_V^\epsilon$, the length each component of $\partial V$ on $a$ is at most $\epsilon$;
\item for $a \in [x,y]\setminus \I_V^\epsilon$, some component $\beta$ of $\partial V$ has $l_a(\beta)\geq \epsilon'$;
\item for $a,b$ in the same component of $[x,y]\setminus \I_V^\epsilon$, we have $d_V(a,b)< \M_\epsilon$; and 
\item if $V\pitchfork W$ then $\I_V^\epsilon\cap \I_W^\epsilon=\emptyset$. 
\end{itemize}
This $\I_V^\epsilon$ is called the {\em $\epsilon$--thin interval} for $V$, or just the \emph{thin interval} when $\epsilon$ is understood. 
While  $[x,y]$ is suppressed in the notation, the geodesic with respect to which the interval $\I_V^\epsilon$ is defined should be clear from context. 

\begin{rem}
Note that for us $\I_V^\epsilon$ is a segment in Teichm\"uller space, whereas Rafi works with the corresponding time interval $I\subset \R$. 
We also  caution that $\I_V$ is not necessarily the same as the ``active interval'' for $V$ considered by Rafi in \cite{Ra}, where it is additionally required that the restriction of $[x,y]$ to $V$ behaves like a unit-speed 
Teichm\"uller geodesic in $\T(V)$.
\end{rem}

If $\I_V^\epsilon\neq\emptyset$ we will say that $V$ {\em becomes thin} along $[x,y]$. In particular if  $d_V(x,y)\geq \M^\epsilon$, then $\I_V^\epsilon\neq \emptyset$ and so $V$ becomes thin along $[x,y]$. Note that the second condition above says that the complement of the union of thin intervals (for all proper subsurfaces) lies in the $\epsilon'$--thick part of $\T(S)$.

We always assume that $\M_\epsilon$ is chosen large enough to be a valid threshold in in the distance formula \eqref{dist}. In the case $\epsilon=\epsilon_0$ we will omit the parameter and simply write $\M$ and $\I_V$ for $\M_{\epsilon_0}$ and $\I_V^{\epsilon_0}$. Thus $\M$ is a global constant that depends only on the topology of $S$.

%%%
\subsection{Reverse triangle inequality}
We will repeatedly use the fact that the projection of a Teichm\"uller geodesic to the curve complex of any subsurface other than an annulus forms an unparameterized quasi-geodesic that, in particular, does not backtrack. This phenomenon is captured by the following ``reverse triangle inequality,'' which was proved first in the case of the curve complex of the whole surface by Masur--Minsky \cite{MM} and then for general subsurfaces by Rafi \cite[Thm B]{Ra}. 

\begin{lemma}[Reverse triangle inequality]
\label{lem:subsurfacequasi}
There exists a global constant $\B>0$ such that for any non-annular subsurface $V$ (including $S$ itself)  and for any geodesic interval $[x,y]\subset\T(S)$ and any point $a\in [x,y]$ we have 
\begin{align}\label{eqn:reverseTriangle}
d_V(x,a)+d_V(a,y) &\leq d_V(x,y)+\B.
\end{align}
\end{lemma}

In the exceptional annulus case, Rafi \cite{Ra} shows that the reverse triangle inequality for $d_{\CC(A)}$ does hold when the twisting distance is measured with respect to the quadratic differential defining the geodesic. However, it is unknown whether the reverse triangle inequality holds with twisting defined in terms of the hyperbolic metric, as it is in this paper. So, to deal with this exceptional case we instead appeal to the following result, 
still following Rafi \cite{R05}, (c.f., Theorem 5.5 of \cite{Ra}):  though projection to the annulus may not be large between the points we consider, we find a subsurface that does register a large projection distance.

\begin{lemma}[R.T.I. exception]
\label{lem:annulusquasi}\label{exception}
For any sufficiently large $M'$, there exists $\epsilon'>0$ with the following property. Suppose that a simple curve $\alpha$ on $S$ satisfies $l_a(\alpha)\leq\epsilon'$ for some point $a\in [x,y]$ with $d_S(a,x),d_S(a,y)\geq 6+\B$. Then there exists a subsurface $Z\subsetneq S$ disjoint from $\alpha$ (possibly the annulus with core curve $\alpha$) for which $d_Z(x,y)> M'$.
\end{lemma}
\begin{proof}
Let $\nu^\pm$ denote horizontal and vertical foliations for the Teichm\"uller geodesic $[x,y]$. By choosing $\epsilon'$  small, thus forcing $1/l_a(\alpha)$ to be large, Theorem 5.6 of \cite{R05} ensures that we can find a component $Y$ of $S\setminus \alpha$ (possibly the annulus with core curve $\alpha$) for which the intersection number
\[i_Y(\nu^+_Y,\nu^-_Y)\]
is as large as we like (see \S2 of \cite{R05} for the definitions of $i_Y$ and of the projections $\nu^\pm_Y$ to the ``arc and curve complex'' of $Y$). As is well known, the two arc systems $\nu^\pm_Y$ can only have large intersection number if the projection $d_Z(\nu^+_Y,\nu^-_Y)$ to some subsurface $Z\subset Y$ is large. 
Thus, by choosing $\epsilon'$ sufficiently small, we may assume that there is a subsurface $Z\subset S$ disjoint from $\alpha$ (possibly the annulus with core curve $\alpha$) for which $d_Z(\nu^+,\nu^-)> M'+2\M$. It follows that $Z$ determines a nonempty thin interval $\I_Z$ along the bi-infinite geodesic $[\nu^-,\nu^+]$. Moreover, since $d_S(\alpha,\partial Z) \leq 1$, the reverse triangle inequality implies that for any $t\in \I_Z$ and any $b\in [t,a]$ we have
\[d_S(a,b) \leq d_S(a,b) + d_S(b,t) \leq d_S(a,t) + \B \leq 5 + \B.\]
Thus it must be the case that $\I_Z \subset[x,y]$, for otherwise either $d_S(x,a)$ or $d_S(y,a)$ would be smaller than $5 + \B$, which is not the case. Therefore the projection to $Z$ changes by at most $\M$ outside of $[x,y]$, and so we conclude that $d_Z(x,y)> M'$ as desired.
\end{proof}

Going forward, for each $\epsilon\leq \epsilon_0$ we additionally assume that $\M_\epsilon$ is chosen large enough to satisfy $\M_\epsilon \geq \B$ and so that Theorem~\ref{exception} above applies with $M' = \M_\epsilon$.

%%%%%%%%%%%%%%%%%%%%%%%%%%%
\section{The geometry of statistically thick geodesics}
\label{sec:thick_geodesics}
%%%%%%%%%%%%%%%%%%%%%%%%%%%%
In the study of Teichm\"uller geometry, one finds that the thick part $\T_\epsilon$ behaves very much like a negatively-curved space. For example in Theorems 4.4 and 7.6 of \cite{KL} Kent and Leininger show that geodesic triangles contained entirely within $\T_\epsilon$ are $\delta$--thin for some $\delta$ depending on $\epsilon$, 
and that the projection of any geodesic $\gamma\subset\T_{\epsilon}$ to the curve complex $\CC(S)$ is an honest \emph{parametrized} quasi-geodesic which, in particular, must progress at a linear rate. These facts can also be deduced from the distance formula (\ref{dist}) together with properties of thin intervals $\I_V^\epsilon$ (\S\ref{sec:thinness} above).

All of these negative-curvature properties are lost when geodesics are allowed to enter the thin part. For example, Minsky's product region theorem \cite{Min-PR} shows that geodesic triangles in $\T(S)\setminus \T_\epsilon$ need not be $\delta$--thin for any $\delta$, and it is easy to construct arbitrarily long geodesics in $\T(S)\setminus\T_\epsilon$ that project to uniformly bounded diameter sets in $\CC(S)$.

However, each of these extremes---living entirely in $\T_\epsilon$ or entirely in $T(S)\setminus\T_\epsilon$---is quite rare, as a typical Teichm\"uller geodesic will spend part of its time in $\T_\epsilon$ and part of its time in 
$\T(S)\setminus \T_\epsilon$. In this section we develop tools to study geodesics with exactly this behavior (later on, in \S\ref{sec:separationAndThickness} we will show that this behavior is in fact generic in a certain quantifiable sense).

Our techniques rely on controlling the fraction of time a Teichm\"uller geodesic spends in a given thick part $\T_\epsilon$.
We call this quantity the \emph{thick-stat}; for a nondegenerate geodesic segment $[x,y]\subset \T(S)$ it is denoted by
\[\thickfrac_{\epsilon}[x,y]= \frac{\bigl| \{ 0 \le s \le d_\T(x,y) : y_s \in \T_{\epsilon}\}\bigr|}{d_\T(x,y)},\]
where $y_s$ is the time--$s$ point on the geodesic ray from $x$ through $y$. Thus $[x,y]\subset \T_\epsilon$ is equivalent to $\thickfrac_\epsilon[x,y]=1$. In the following subsections, we will show that negative-curvature properties similar to those mentioned above for thick geodesics also hold when $\thickfrac_\epsilon$ is merely bounded away from zero.

%%%%%%%%%%%%%
\subsection{Progress in the curve complex}
\label{s:def-prog}

The goal of this subsection is to prove Theorem~\ref{prop:def-prog}, which says that geodesics that spend a definite fraction of their time in the thick part $\T_\epsilon$ must move at a definite linear rate in the curve complex.

The idea is that long subintervals contained in $\T_{\epsilon}$ contribute to progress in $\CC(S)$; alternately, one could consider intervals in the complement of all the $\epsilon$--thin intervals $\I_V^\epsilon$ for proper subsurfaces $V\subsetneq S$. 
For this analysis, we would like to bound the number of connected components of $\bigcup_V \I_V^\epsilon$ in terms of $d_S(x,y)$. One bound is given by the number of nonempty thin intervals. While there may be arbitrarily many such $\I_V^\epsilon$, some of these will be redundant in the sense that $\I_V^\epsilon\subset \I_W^\epsilon$ for some other subsurface $W$.

\begin{definition}[Thin-significance]
Fix $0 < \epsilon\leq \epsilon_0$. A proper subsurface $V\subsetneq S$ is said to be {\em $\epsilon$--thin-significant} (or simply {\em thin-significant})
for the geodesic segment $[x,y]$ if $d_{\CC(V)}(x,y)\geq 3\M_\epsilon$ and for every other proper subsurface $Z\subsetneq S$ with $d_{\CC(Z)}(x,y)\geq 3\M_\epsilon$ we have $\I_V^\epsilon\not\subset \I_Z^\epsilon$.
\end{definition}

\begin{rem*}
In this subsection we will focus on the curve complex distance $d_{\CC(V)}$ for a subsurface $V$. Recall that this agrees with the usual projection distance $d_V$ in the case that $V$ is non-annular, but that $d_{\CC(A)}$ and $d_A$ differ for annuli.
We will take care to handle exceptional annuli carefully.
\end{rem*}

Our first goal is to bound the number of thin-significant subsurfaces along an arbitrary geodesic. For this, we will use the work of Rafi--Schleimer \cite{RS} bounding the size of an antichain in the poset of subsurfaces of $S$.
 
\begin{definition}[Antichain]  Given a subsurface $\Sigma\subset S$ a pair of points $x,y\in\T(S)$ and constants $T_1\geq T_0>0$, a collection $\Omega$ of proper subsurfaces of $\Sigma$ is an \emph{antichain}  for 
$(\Sigma,x,y,T_0,T_1)$
if the following hold:
\begin{itemize}
\item if $Y,Y'\in\Omega$, then $Y$ is not a proper subsurface of $Y'$;
\item if $Y\in \Omega$, then $d_{\CC(Y)}(x,y)\geq T_0$; and 
\item if $Z\subsetneq \Sigma$ and $d_{\CC(Z)}(x,y)\geq T_1$, then $Z\subset Y$ for some $Y\in\Omega$.
\end{itemize}
\end{definition}

\begin{lemma}[Antichain bound {\cite[Lem 7.1]{RS}}]
\label{lem:chain}
For every $\Sigma\subset S$ and sufficiently large $T_1\geq T_0>0$, there is a constant $A=A(\Sigma,T_0,T_1)$ so that if $\Omega$ is an
antichain for $(\Sigma,x,y,T_0,T_1)$ then
\[ \abs{\Omega} \leq A\sdot d_{\CC(\Sigma)}(x,y).\]
\end{lemma}

We now prove a proposition showing that if there are a large enough number of thin-significant subsurfaces 
along a geodesic, then the image of the geodesic makes definite progress in the curve complex. The following notation will be used in the proof.

\begin{definition}
Consider a geodesic segment $[x,y]\subset \T(S)$ and a collection $\Omega$ of proper subsurfaces of $S$. We will consider three partial orders on the set $\Omega$:
\begin{enumerate}
\item $V\leq_1 W \iff V\subset W$,
\item $V\leq_2 W \iff \I_V^\epsilon \subset \I_W^\epsilon$, and 
\item $V\leq_3 W \iff V\subset W$ and $\I_V\subset \I_W$.
\end{enumerate}
The subcollection of $\Omega$ consisting of maximal elements with respect to $\leq_*$ will be denoted $(\Omega)_*$; notice that these sets are related by $(\Omega)_1\subset (\Omega)_3 \supset (\Omega)_2$. Elements of $(\Omega)_1$ are said to be \emph{topologically maximal} with respect to $\Omega$.
\end{definition}

\begin{proposition}[Progress from thin-significant subsurfaces] 
\label{boundmax}
\label{prop:maximal}
For any positive $\epsilon \leq \epsilon_0$ and any $t_0$, there is a constant $N$ such that if $d_{\CC(S)}(x,y)\leq t_0$, then the number of
$\epsilon$--thin-significant subsurfaces along $[x,y]$ is at most $N$.
\end{proposition}
\begin{proof}
Let $\Omega = \{V\subsetneq S : d_{\CC(V)}(x,y)\geq 3\M_\epsilon\}$ be the collection of proper subsurfaces which have a large projection. By definition, the set of $\epsilon$--thin-significant subsurfaces is exactly given by $(\Omega)_2$. On the other hand, the subcollection $(\Omega)_1$ of topologically maximal subsurfaces clearly forms an antichain for  $(S,x,y,3\M_\epsilon,3\M_\epsilon)$. By Lemma~\ref{lem:chain}, we therefore have $\abs{(\Omega)_1} \leq At_0$ for some constant $A$. We will extend this to a bound on the cardinality of the larger set $(\Omega)_3$; this will imply the proposition because $(\Omega)_2\subset(\Omega)_3$.

Fix a proper subsurface $W\in \Omega$ and consider the set $\mathcal{U}_W = \{V\in (\Omega)_3 : V\subsetneq W\}$. We claim that $\abs{\mathcal{U}_W}$ is bounded by a constant depending only on the complexity of $W$. By the above, this will suffice because each $V\in(\Omega)_3$ is either equal to or properly contained in some topologically maximal proper subsurface $W\in (\Omega)_1$.

First consider those $V\in\mathcal{U}_W$ for which $\I_V^\epsilon\cap \I_W^\epsilon\neq \emptyset$. The definition of $\leq_3$ implies that $\I_V^\epsilon\not\subset \I_W^\epsilon$; therefore $\I_V^\epsilon$ must overlap with at least one endpoint of $\I_W^\epsilon$. If $\I_{V_1}^\epsilon$ and $\I_{V_2}^\epsilon$ both contain the initial endpoint of $\I_W^\epsilon$, then $\I_{V_1}^\epsilon\cap \I_{V_2}^\epsilon\neq \emptyset$ and so we cannot have $V_1\pitchfork V_2$. Since there is a universal bound on the number of subsurfaces such that no two intersect transversely, this bounds the number of $V\in \mathcal{U}_W$ for which $\I_V^\epsilon\cap \I_W^\epsilon\neq \emptyset$.

It remains to bound the number of $V\in \mathcal{U}_W$ for which $\I_V^\epsilon\cap \I_W^\epsilon = \emptyset$; we will only focus on the case that $\I_V^\epsilon$ occurs before $\I_W^\epsilon$ when traveling from $x$ to $y$. Suppose that $\I_W^\epsilon = [a,b]\subset [x,y]$ and consider the set 
\begin{align*} 
\Omega' &= \{V\in \Omega : V\subsetneq W\text{ and }d_{\CC(V)}(x,a) \geq 2\M_\epsilon\}\\
&\quad \cup \{A \text{ an annulus} : A\subsetneq W,\text{ and } d_{\CC(A)}(x,a)\geq 4\M_\epsilon\}.
\end{align*}
 We claim that the subcollection $(\Omega')_1$ forms an antichain for $(W,x,a,2\M_\epsilon,4\M_\epsilon)$: The only issue  is to check that every $Y\subsetneq W$ with $d_{\CC(Y)}(x,a)\geq 4\M_\epsilon$ is contained in an element of $(\Omega')_1$. If $Y$ satisfies  the reverse triangle inequality then  $d_{\CC(Y)}(x,y)\geq d_{\CC(Y)}(x,a)-\B \geq 3\M_\epsilon$ and therefore  $Y\in \Omega'$. If $Y$ does not satisfy the reverse triangle inequality it is an annulus and automatically $Y\in\Omega'$.  Since $d_{\CC(W)}(x,a)\leq \M_\epsilon$, Lemma~\ref{lem:chain} now gives a bound on $\abs{(\Omega')_1}$. 

Finally, notice that for each $V\in \mathcal{U}_W$ with $\I_V^\epsilon$ occurring before $\I_W^\epsilon$ along $[x,y]$, the triangle inequality gives $d_{\CC(V)}(x,a)\geq d_{\CC(V)}(x,y)-\M_\epsilon \geq 2\M_\epsilon$ and so ensures that $V\in \Omega'$. Therefore each such $V$ is contained in some topologically maximal $Z\in \Omega'$; that is to say, each $V\in \mathcal{U}_W$ with $\I_V^\epsilon$ occurring before $\I_W^\epsilon$ along $[x,y]$ is contained in $\mathcal{U}_Z$ for some $Z\in (\Omega')_1$. The bound on $\abs{\mathcal{U}_W}$ now follows by the bound on $\abs{(\Omega')_1}$ and induction on the complexity of the subsurface $W$. 
\end{proof}

Proposition~\ref{prop:maximal} gives control on the union of the $\epsilon$--thin intervals $\I_V^\epsilon$ for all subsurfaces $V$ with $d_{\CC(V)}(x,y)$ large. However, the potential failure of the reverse triangle inequality enables $[x,y]$ to contain many thin intervals that are not accounted for by Proposition~\ref{prop:maximal}. The following lemma allows us to control these as well by extending the thin intervals so as to have certain large projections 
that will help with our bookkeeping.

\begin{lemma}[Extended thin interval]
\label{lem:extended_thin}
For any $\epsilon \leq \epsilon_0$ and $t_0 > 0$, there exists $M'>0$ with the following property. If a subsurface $W\subsetneq S$ satisfies $d_W(x',y') > M'$ for some pair of points $x',y'\in [x,y]$, then there is a connected interval $[a,b]\subset[x,y]$ containing $\I_W^\epsilon$ and contained entirely in the $\epsilon$--thin part of $\T(S)$ such that either
\begin{itemize}
\item $d_S(a,b) \geq t_0 + 3\B$ (and thus also $d_S(x,y)> t_0$ by Lemma~\ref{lem:subsurfacequasi}), or
\item $[a,b]$ nontrivially intersects the $\epsilon$--thin interval $\I_V^\epsilon\subset [x,y]$ of some subsurface $V\subsetneq S$ satisfying $d_V(x,y)\geq 3\M_\epsilon$.
\end{itemize}
We call such an interval $J=[a,b]$ an \emph{extended $\epsilon$--thin interval for $W$}.
\end{lemma}
\begin{proof}
Let $K$ be the implied constant in the distance formula \eqref{dist} corresponding to the threshold $5\M_\epsilon$. For this $K$ and the given $t_0$, set
\[M' = K^2(t_0 + 3\B) + K^2 + K + 2\M_\epsilon.\]
Now suppose that $W$ determines a thin interval $\I_W^\epsilon\subset[x,y]$ of length $L\geq 0$ (where we allow the possibility that $\I_W^\epsilon=\emptyset$ and $L = 0$). Since the projection to $W$ can change by at most $\M_\epsilon$ outside of $\I_W^\epsilon$, the distance formula implies that
\[d_W(x',y')  \leq 2\M_\epsilon + KL + K\]
for any $x',y'\in [x,y]$. Therefore, the hypothesis $d_W(x',y')\geq M'$ on $W$ ensures that
\[L \geq (M' - 2\M_\epsilon - K)/K \geq K(t_0 + 3\B) + K.\]
In particular $\I_W^\epsilon\neq\emptyset$. Thus we have shown that there exists a nonempty interval $J=[a,b]\subset [x,y]$ (for example, $\I_W^\epsilon$ itself) that 
\begin{enumerate}
\item contains $\I_W^\epsilon$, and
\item is the union of finitely many nonempty thin intervals $\I_V^\epsilon\subset [x,y]$ for proper subsurfaces $V\subsetneq S$ (and so is entirely contained in the $\epsilon$--thin part of $\T(S)$).
\end{enumerate}
We claim that for any such interval $J$ that fails to satisfy the conclusion of the lemma, there exists a \emph{strictly} larger subinterval $J' \supsetneq J$ that again satisfies (1)--(2). As there are only finitely many subintervals satisfying (2) (since at most finitely many curves become shorter than $\epsilon$ along the compact segment $[x,y]$), repeated applications of the claim will eventually produce the desired subinterval.

To prove the claim, we may suppose that $J=[a,b]$ satisfies (1)--(2) above and fails to meet the conclusion of the lemma. In that case, there necessarily exists a nonempty collection $\Omega$ of subsurfaces $V\subsetneq S$ for which $d_{V}(a,b)\geq 5\M_\epsilon$, for otherwise the distance formula (applied with threshold $5\M_\epsilon$) would give 
\[d_\T(a,b) \leq K d_S(a,b) + K < K(t_0 + 3\B) +K \leq L,\]
contradicting the assumption $[a,b]\supset \I_W^\epsilon$ (recall that $L$ is the length of $\I_W^\epsilon$).

Note that the property $d_V(a,b) \geq 5\M_\epsilon$ implies that the thin interval $\I_V^\epsilon\subset[x,y]$ of each $V\in \Omega$ nontrivially intersects $[a,b]$. If any such interval $\I_V^\epsilon$ were completely contained within $[a,b]$ then, since the projection to $V$ can move at most $\M_\epsilon$ outside of $\I_V^\epsilon$, the triangle inequality would give $d_V(x,y)\geq 3\M_\epsilon$. As this is evidently not the case (since $J$ fails to satisfy the conclusion of the lemma), it must be that the thin interval $\I_V^\epsilon$ of each $V\in \Omega$ nontrivially intersects the complement of $[a,b]$ as well. Therefore, 
\[J' := [a,b]\cup \bigcup_{V\in \Omega} \I_V^\epsilon\]
is a connected subinterval of $[x,y]$ that properly contains $J$ and again satisfies (1)--(2). Thus the claim holds and so the lemma is verified.
\end{proof}

By the distance formula (\ref{dist}), any long interval disjoint from all thin intervals along $[x,y]$ must travel a large distance in the curve complex $\CC(S)$ of the whole surface. The following lemma says that each such subinterval contributes to the curve complex distance along the total geodesic.

\begin{lemma}[Cumulative contribution of subintervals]
\label{lem:bigSubintervalsAdd}
There exist constants $0<\rho_1<1$ and $D_1>0$ such that for all $d>D_1$, 
 if $[x,y]$ is a Teichm\"uller geodesic that contains $n$ subintervals $[x_i,y_i]$ with disjoint interiors whose endpoints 
 satisfy $d_S(x_i,y_i)\geq d$, then 
$$d_S(x,y)\geq \rho_1 nd.$$
\end{lemma}
\begin{proof}
Applying the reverse triangle inequality (\ref{eqn:reverseTriangle}) to the points $x_i$ and $y_i$ we have $d_S(x,x_i) + d_S(x_i,y_i)+d_S(y_i,y) \leq d_S(x,y)+2\B$. By recursively applying this observation to $[x,x_i]$ and $[y_i,y]$ and then throwing out the complementary intervals, we find that
\[d_S(x,y)\geq \sum d_S(x_i,y_i)-2n\B\geq nd-2n\B.\]
Choose $D_1>4\B$ and $\rho_1=1/2$.  Then for $d\geq D_1$ the quantity on the right side is at least 
$\rho_1 nd$. 
\end{proof}

We now fix once and for all a ``definite progress'' constant $\D>0$, sufficiently large so that $\rho_1 \D>D_1$ (and thus $\D>D_1$ as well), and make the following definition.

\begin{definition}
\label{def:primary_thin} 
For any $\epsilon \leq \epsilon_0$, set $t_0=\rho_1 \D$ and let $M' = M'(\epsilon)$ be the corresponding constant provided by Lemma~\ref{lem:extended_thin}. Then define the \emph{primary $\epsilon$--thin portion} $\mathcal{W}_\epsilon$ of a geodesic segment $[x,y]$ to be the union of $\epsilon$--thin intervals $\I_V^\epsilon\subset[x,y]$ for all proper subsurfaces with $d_{\CC(V)}(x,y)\geq 3\M_\epsilon$ together with an extended $\epsilon$--thin interval 
$J_W^\epsilon\subset[x,y]$ for any proper subsurface $W$ satisfying $d_W(x',y')\geq M'$ for some pair of points $x',y'\in [x,y]$.
\end{definition}

\begin{lemma}[Primary thin portion]
\label{lem:primary_thin}
The primary $\epsilon$--thin portion $\mathcal{W}_\epsilon$ is completely contained in the $\epsilon$--thin part of $\T$. Furthermore, if $d_S(x,y)\leq \rho_1 \D$, then the number of connected components of $\mathcal{W}_\epsilon$ is bounded by a constant $N'$ depending only on $\epsilon$.
\end{lemma}
\begin{proof}
The first assertion is immediate since $\mathcal{W}_\epsilon$ is a union of $\epsilon$--thin intervals. For the second assertion, note that the union $\mathcal{W}'$ of $\epsilon$--thin intervals $\I_V^\epsilon$ for all proper 
subsurfaces with $d_{\CC(V)}(x,y) \geq 3\M_\epsilon$ has a bounded number of connected components by Proposition~\ref{prop:maximal} (since passing to thin-significant subsurfaces does not change the union).  
Consider now an extended $\epsilon$--thin interval $J_W^\epsilon$ contributing to $\mathcal{W}_\epsilon$. Since $d_S(x,y)\leq \rho_1\D = t_0$, Lemma~\ref{lem:extended_thin} implies that $J_W^\epsilon$ intersects 
$\I_V^\epsilon$ for 
some subsurface $V$ with $d_V(x,y)\geq 3\M_\epsilon$. We claim that either $d_{\CC(V)}(x,y)\geq 3\M_\epsilon$, so that $\I_V^\epsilon \subset \mathcal{W}'$ and thus 
$J_W^\epsilon\cap \mathcal{W}'\neq \emptyset$, or else $V$ is an annulus 
with $\min(l_x(\partial V),l_y(\partial V))< \epsilon_0$. Since there can be at most $6g-6$ such annuli and $\mathcal{W}'$ has a bounded number of components, this will suffice. 

If $V$ is non-annular, then $d_{\CC(V)} = d_V$ and the claim is immediate. Otherwise $V$ is an annulus with $d_V(x,y)\geq 3\M_\epsilon$. Since  $\M_\epsilon$ may be assumed large enough to satisfy the universal condition 
$\M_\epsilon \geq 36\log_+(1/\epsilon_0)+ 6$, Lemma~\ref{lem:twist+lenVShyp} implies that either $\min(l_x(\partial V),l_y(\partial V)) < \epsilon_0$ or else $d_{\CC(V)}(x,y) \geq e^{\M_\epsilon/2} \geq 3\M_\epsilon$, as claimed.
\end{proof}

For any interval $[a,b]\subset [x,y]\setminus \mathcal{W}_{\epsilon}$ in the complement of the primary $\epsilon$--thin portion, the construction of $\mathcal{W}_\epsilon$ ensures that $d_W(a,b)\leq M'$ for all proper subsurfaces $W$ of $S$. Applying the distance formula (\ref{dist}) with $M' = M'(\epsilon)$ as the threshold, we now see that $d_\T(a,b)\lboth_\epsilon d_S(a,b)$ for any connected interval $[a,b]\subset [x,y]\setminus\mathcal{W}_\epsilon$. This gives rise to a fixed value $\L_\epsilon$ such that any interval $[a,b]$ of length at least $\L_\epsilon$ that lies entirely 
in  $[x,y]\setminus\mathcal{W}_\epsilon$ satisfies $d_S(a,b)\geq \D$. Thus according to Lemma~\ref{lem:bigSubintervalsAdd}, if $I$ is any interval along a geodesic that contains 
a subinterval of length $\L_\epsilon$ that is disjoint from $\mathcal{W}_\epsilon$, then the distance in the curve complex between the endpoints of $I$ is at least $\rho_1\D$. We now come to the main result of this subsection.

\begin{theorem}[Definite progress]\label{prop:def-prog}
For every $\epsilon > 0$ and $0 < \theta < 1$, there exists a constant $R_1>0$ such that 
\[d_S(x,y) \gmul_{\epsilon,\theta} d_\T(x,y)\]
for every Teichm\"uller geodesics $[x,y]$ satisfying $d_\T(x,y)\geq R_1$ and $\thickfrac_\epsilon[x,y]\geq \theta$. 
\end{theorem}

\begin{proof}
Since shrinking $\epsilon$ preserves the hypothesis $\thickfrac_\epsilon[x,y]\geq \theta$, we may assume $\epsilon\leq \epsilon_0$. Let $N' = N'(\epsilon)$ denote the constant obtained from Lemma~\ref{lem:primary_thin}. Choose $n$ so that $n\theta > 1$ and make the following definitions:
\[\begin{array}{cccc}
\displaystyle\theta'= \frac{n\theta-1}{n-1}, & \displaystyle T_0\geq\frac{\L_\epsilon(N'+1)}{\theta'}, & \displaystyle R_1 = 2T_0,
& \displaystyle \rho=\frac{\rho_1^2 \D}{2nT_0}.
\end{array}\]
Let $[x,y]$ be a Teichm\"uller geodesic of length 
$r\geq R_1$ satisfying $\thickfrac_\epsilon[x,y]\geq \theta$.
Set $m = \floor{r/T_0}$ and divide $[x,y]$ into $m$ subsegments of length $r/m\geq T_0$. 
Let us say that a subsegment $[a,b]\subset [x,y]$ is {\em stalled} if $d_S(a,b) < \rho_1 \D$ and 
{\em progressing} if 
$d_S(a,b)\ge \rho_1 \D$.  Suppose that $m_1$ of the subsegments are stalled, and thus $m_2=m-m_1$ are progressing.

Given a stalled segment $[a,b]$, we decompose it into its primary $\epsilon$--thin portion $\mathcal{W}_\epsilon$ and note that, since it is stalled, Lemma~\ref{lem:primary_thin} ensures $\mathcal{W}_\epsilon$ has at most $N'$ connected components. Therefore we conclude that $\mathcal{W}_\epsilon$ has at most $N'+1$ complementary subintervals in $[a,b]$. Furthermore, each complementary subinterval has length at most $\L_\epsilon$, for otherwise we would have $d_S(a,b)\geq \rho_1\D$ by the paragraph preceeding Theorem~\ref{prop:def-prog}. Since $\mathcal{W}_\epsilon$ is contained in the $\epsilon$--thin part, we see that the total amount of time that this interval $[a,b]$ spends in the thick part is at most 
$$(N'+1)\L_\epsilon\leq \theta' T_0 \leq \theta' r/m.$$ 
Therefore the total amount of time that the full interval $[x,y]$ spends in the thick part is at most
$$\left(\theta' \frac rm \right)m_1 + \left(\frac rm \right)m_2 = \frac rm (\theta'm_1+m_2).$$ 

We claim that $m_2\geq m/n$. If this were not the case, then we necessarily have $m_1 > (n-1)m/n$. Since $\theta'<1$, it follows that 
\[\theta' \sdot m_1 + 1\sdot m_2 < \theta' \sdot m\frac{n-1}n + 1\sdot m \frac 1n,\]
where the inequality is valid by 
the elementary fact that for any constants $a,b,c,d,\alpha,\beta$ such that  $a+b=c+d$ and $0<\alpha<\beta$
we have 
\begin{equation}
\label{eq:shifting} \alpha \sdot a + \beta\sdot  b < \alpha\sdot c + \beta\sdot d \iff a>c.
\end{equation}
But then the amount of time that $[x,y]$ is thick is less than 
$$\frac rm \left(\theta' m \frac{n-1}{n} + m\frac{1}{n}\right)
=  r\left(\frac{n\theta-1}{n-1}\cdot \frac{n-1}{n} + \frac{1}{n}\right)
=  r\theta,$$
which contradicts the assumption on $[x,y]$. Therefore $m_2 \geq m/n$, as claimed.

On each of the $m_2$ progressing intervals, 
the curve complex distance between endpoints is at least $\rho_1\D$. 
Therefore, cumulative contribution of subintervals (Lemma~\ref{lem:bigSubintervalsAdd}) implies  that
\[d_S(x,y) \geq \rho_1  m_2 (\rho_1\D) 
\geq \rho_1^2\D \frac mn \geq \frac{\rho_1^2 \D}n \left(\frac{r}{T_0} -1\right) \geq \frac{\rho_1^2 \D}{2nT_0}r=\rho r.\qedhere\]
\end{proof}

\begin{rem}\label{hamenstadt}
After developing our proof of Theorem~\ref{prop:def-prog} we learned of an independent yet closely related result of Hamenst\"adt's, namely Proposition 2.1 of \cite{Ham}, which under the same hypotheses provides a lower bound on $d_S(x,y)$ that is constant rather than linear in $d_\T(x,y)$. In fact, the linear bound in Theorem~\ref{prop:def-prog} may be deduced from Hamenst\"adt's result by breaking $[x,y]$ into subintervals, applying \cite[Proposition 2.1]{Ham} to those with large thick-stat, and adding the resulting contributions using Lemma~\ref{lem:bigSubintervalsAdd}, much as we have done above. With this approach \cite[Proposition 2.1]{Ham} would effectively replace the use of Proposition~\ref{prop:maximal} and Lemma~\ref{lem:extended_thin} in our argument. However, we have decided to retain our original argument using Proposition~\ref{prop:maximal} and Lemma~\ref{lem:extended_thin} as we believe these to be of independent interest.
\end{rem}

%%%%%%%%%%%%%%
\subsection{A statistical thin triangles statement}
\label{s:partially_thick}
%%%%%%%%%%%%%%%%%%%
In this subsection we prove Theorem~\ref{thm:partiallyThickIntervalsAreClose} and obtain thinness results for geodesic triangles whose sides satisfy various thick-stat conditions. Recall that given $\epsilon > 0$ there is a $\delta> 0$ such that every geodesic triangle whose sides live entirely in $\T_\epsilon$ is $\delta$--thin. This fact can be deduced from the following theorem of Rafi, which gives specific information under much more general conditions.

\begin{theorem}[Rafi {\cite[Theorem 8.1]{Ra}}]
\label{thm:RafiThinTriangles}
For every $\epsilon>0$ there exist  constants $C_1,L_1$ such that
if $I\subset [x,y]\subset \T(S)$ is a geodesic subinterval of length at least $L_1$ 
lying entirely in the $\epsilon$--thick part,
then for all $z\in\T(S)$, we have 
\[ I \cap \Nbhd_{C_1}([x,z]\cup[y,z]) \neq \emptyset. \]
\end{theorem}

We weaken the hypothesis  to only require a definite thick-stat.

\begin{theorem:triangles}
For any $\epsilon>0$ and $0<\theta \leq 1$, there exist constants $C,L$ such that 
if $I\subset[x,y]\subset \T(S)$ is a geodesic subinterval of length at least $L$ 
and at least proportion $\theta$ of $I$  is $\epsilon$--thick,
then for all $z\in\T(S)$, we have 
\[ I \cap \Nbhd_C([x,z]\cup[y,z]) \neq \emptyset. \]
\end{theorem:triangles}

Before proving this result, we discuss two consequences. Firstly we observe that there is not merely one point in the subinterval $I$ which is close to $[x,z]\cup[y,z]$, but in fact this conclusion holds for a large fraction of the interval $I$.

\begin{proposition}
For any $\epsilon > 0$ and $0 < \theta' <\theta \leq 1$,
 there are constants $L', C'$ so that if a side $[x,y]$ of a geodesic triangle $\triangle(x,y,z)\subset \T(S)$ contains a subinterval $I\subset [x,y]$ of length at least $L'$ with $\thickfrac_\epsilon(I)\geq \theta$, then at least proportion $\theta'$ of $I$ is within distance $C'$ of $[x,z]\cup[y,z]$. That is, if $\len$ denotes Lebesgue measure along 
 a geodesic segment, 
\[   \len\left(\{ I \cap \Nbhd_{C'}([x,z]\cup[y,z]) \}\right) \ge \theta'\cdot \len(I).
\]
\end{proposition}

\begin{proof}
Let $\rho = \frac{\theta-\theta'}{1-\theta'}$. Apply Theorem~\ref{thm:partiallyThickIntervalsAreClose} to $\T_\epsilon$ with the fraction $\rho$ and let $L'=L$ and $C$ be the corresponding constants. Given a subinterval $I\subset [x,y]$ satisfying the hypotheses of the theorem, divide $I$ into $n = \floor{\len(I)/L}$ subintervals of equal length (the length will be between $L$ and $2L$). Let $a$ denote the fraction of these subintervals that have $\thickfrac_\epsilon \geq \rho$ (so $a = \frac{k}{n}$ for some $k\in \{0,\dotsc,n\}$). Each of these $na$ subintervals can spend at most all of their time in $\T_\epsilon$ and the other $n(1-a)$ subintervals spend less than proportion $\rho$ of their time in $\T_\epsilon$. Therefore the maximum amount of time the whole interval $I$ can spend in $\T_\epsilon$ is less than
\[1\cdot na\cdot\frac{\len(I)}{n} + \rho\cdot n(1-a)\cdot\frac{\len(I)}{n} = \len(I)(a+\rho-\rho a).\]
Since we have $\thickfrac_\epsilon(I)\geq \theta$ by hypotheses, this implies $\theta < a+\rho-a\rho$. That is,
\[a > \frac{\theta-\rho}{1-\rho} = \theta'.\]
Now, Theorem~\ref{thm:partiallyThickIntervalsAreClose} implies that each subinterval with $\thickfrac_\epsilon\geq \rho$ contains a point within distance $C$ of $[x,z]\cup[y,z]$.
Therefore, each of the $na$ subintervals satisfying $\thickfrac_\epsilon\geq \rho$ is contained entirely within the $C'=C+2L$ neighborhood of $[x,z]\cup[y,z]$. As  the union of these $na$ subintervals comprise 
proportion $a > \theta'$  of the interval $I$, the statement follows.
\end{proof}

From this we obtain the following immediate corollary.
\begin{corollary}[Statistically thin triangles]
\label{cor:halfThinTriangles}
For all $\epsilon > 0$ and $0 < \theta' < \theta\leq 1$ there exists a constant $\delta$ with the following property. For any geodesic triangle in $\T(S)$ whose three sides have $\thickfrac_\epsilon \geq \theta$, at least 
proportion $\theta'$ of each side of the triangle is contained within $\delta$ of the union of the other two sides.
\end{corollary}

We now give the proof of Theorem~\ref{thm:partiallyThickIntervalsAreClose}.
\begin{proof}[Proof of Theorem~\ref{thm:partiallyThickIntervalsAreClose}]
We will find $L$ so that the conclusion of the theorem applies to any subinterval $I$ with $L\leq \len(I)\leq 2L$ and $\thickfrac_\epsilon(I)\geq \theta$. This will suffice because any long interval with $\thickfrac_\epsilon\geq \theta$ can be partitioned into subintervals satisfying this length condition, one of which must have $\thickfrac_\epsilon \geq \theta$.

Recall that the (coarsely defined) projection $\pi_S\colon \T(S)\to \CC(S)$ sends a point $w\in \T(S)$ to the set of simple closed curve in the Bers marking $\mu_w$ (which is a set of diameter $2$ in the curve complex). The work of Masur--Minsky \cite{MM} shows that there are universal constants $K,C$ so that Teichm\"uller geodesics project to (unparametrized) $(K,C)$--quasi-geodesics under $\pi_S$. By the hyperbolicity of $\CC(S)$ (see \S\ref{sec:curve_graph}), each quasi-geodesic fellow travels any geodesic with the same endpoints, and so there is a constant $\tau > 0$ so that every $(K,C)$--quasi-triangle in $\CC(S)$ is $\tau$--thin. We may furthermore assume that $\tau \geq \B$.

By Theorem~\ref{prop:def-prog}  the interval $I$ moves a definite amount in $\CC(S)$; that is, by making $L$ large, we can arrange for $I$ to project to an arbitrarily long subsegment of the (unparametrized) quasi-geodesic $\pi_S([x,y])$. In particular, by choosing $L$ sufficiently large, we can ensure that there is a point $w\in I$ so that either
\[d_S(\mu_w, \pi_S([y,z])) \geq 2\tau + 6 \qquad \text{or}\qquad d_S(\mu_w,\pi_S([x,z])) \geq 2\tau+6.\]
By choosing such $w$ with $\mu_w$ near the center of $\pi_S(I)$, we can moreover ensure that 
\[d_S(w,t)\geq 2\tau + 6\]
for all points $t\in [x,y]$ outside of $I$, and in particular that $d_S(w,x),d_S(w,y)\geq 2\tau +6$.

Assuming without loss of generality that 
\[d_S(\mu_w,\pi_S([y,z]))\geq 2\tau+6 > \tau,\]
hyperbolicity implies that there is a point $u\in [x,z]$ so that $d_S(w,u)\leq \tau$. We will show that $d_V(w,u)\lboth_{\epsilon,\theta}1$ for all proper subsurfaces $V\subset S$. The result will then follow from the distance formula (\ref{dist}). We first establish the following

\begin{claim}\label{claim:side_projection_bounds}
There is a constant $M_0$ (depending only on $\epsilon$ and $\theta$) so that for any proper subsurface $V\subset S$ satisfying $d_S(\partial V, \mu_w)\leq 2\tau + 3$ we have
\[d_V(x,y), d_V(y,z), d_V(x,z) \leq M_0.\]
\end{claim}

To see this, first observe that for any such $V$ the triangle inequality implies $d_S(\partial V,\pi_S([y,z]))\geq 3$. Therefore $V$ does not become thin along $[y,z]$ and so we may conclude $d_V(y,z)\leq \M$. If $V$ does not become thin along $[x,y]$, then we have the same bound on $d_V(x,y)$. However, $V$ may become thin along $[x,y]$ in which case there is a point $t\in [x,y]$ at which the length of $\partial V$ is smaller than $\epsilon_0$.  Therefore $\mu_t$ contains $\partial V$, which implies
\[d_S(w,t)\leq 2\tau + 5 < 2\tau + 6\]
and consequently that $t\in I\subset [x,y]$. Thus the entire thin interval $\I_V$ for $V$ is contained within $I$ and in particular has length at most $\len(I)\leq 2L$. Since the projection $\pi_V$ to $\CC(V)$ is a Lipschitz map and, up to an additive error, the projection of $[x,y]$ to $\CC(V)$ can only change in the thin interval $\I_V$ (see \S\ref{sec:thinness}), we conclude that $d_V(x,y)$ is bounded in terms of $L$ (and $L$ depends only on $\epsilon$ and $\theta$). Finally, the triangle inequality and the above bounds on $d_V(y,z)$ and $d_V(x,y)$ together provide a uniform bound on $d_V(x,z)$. This completes the proof of Claim~\ref{claim:side_projection_bounds}.

We now show that $d_V(w,u)$ is uniformly bounded for all proper subsurfaces. Consider any $V$ with $d_V(w,u)\geq \M$. Then $V$ becomes thin along $[w,u]$ and so $\partial V$ lies within distance $\tau+2$ of the $\CC(S)$--geodesic from $\mu_w$ to $\mu_u$. In particular
\[d_V(\partial V,\mu_w)\leq \tau+2+d_S(w,u)\leq 2\tau+2\]
and so Claim~\ref{claim:side_projection_bounds} implies that $d_V(x,y)$ and $d_V(x,z)$ are at most $M_0$.

If $V$ is a non-annular surface, then the reverse triangle inequality, applied to $[x,y]$ and $[x,z]$, yields bounds on $d_V(x,w)$ and $d_V(x,u)$ so that we may bound $d_V(w,u)$ by the triangle inequality. 

If $V$ is an annulus whose core curve $\alpha = \partial V$ satisfies $l_w(\alpha) > \epsilon_0$, then $w$ cannot be contained in the (possibly empty) thin interval $\I_V\subset[x,y]$. Thus at least one of the intervals $[x,w]$ or $[w,y]$ is disjoint from $\I_V$ and consequently has $d_V$--projection at most $\M$. Thus we may conclude $d_V(x,w)\leq \M+ M_0$ by the triangle inequality. 
If $l_w(\alpha) \leq \epsilon_0$, then $\mu_w$ necessarily contains $\alpha$ and we instead appeal to Lemma~\ref{lem:annulusquasi}.
According to that theorem applied to $2M_0$, there is a constant $\epsilon'$ such that if $l_w(\alpha) \leq \epsilon'$, then there exists a subsurface $Z$ with $d_S(\partial Z,\alpha)\leq 1$ so that $d_Z(x,y)\geq 2M_0$. Since this contradicts Claim~\ref{claim:side_projection_bounds}, we must in fact have $l_w(\alpha) > \epsilon'$. As above, it follows that $w$ cannot be contained in the $\epsilon'$--thin interval $\I_V^{\epsilon'}\subset[x,y]$ and thus that $d_V(x,w)\leq M_0 + \M_{\epsilon'}$ by the triangle inequality and the theory of thin intervals.

Since $d_S(x,w),d_S(w,z)\geq 2\tau+6$ and $d_S(w,u)\leq \tau$ by assumption, we also have $d_S(x,u),d_S(u,z)\geq \tau+6 \geq \B+6$ by the triangle inequality. Therefore we may apply the same argument, using Lemma~\ref{lem:annulusquasi} as needed, to obtain a bound on $d_V(x,u)$ as well. The triangle inequality then gives the desired bound on $d_V(w,u)$.
\end{proof}

%%%%%%%%%
\section{Comparing measures}
\label{s:measures}

To address genericity and averaging questions, one of course needs to consider a measure. 
In the present context of metric geometry, it is perhaps most natural to consider Hausdorff measure of the appropriate dimension. 

\begin{definition}[Hausdorff measure]
 The $n$--dimensional Hausdorff measure on a metric space will be denoted by $\hausmeas$.  It is defined by
$$\hausmeas(E):= \lim_{\delta\to 0} \left[   \inf \sum \diam(U_i)^n \right],$$
where the infimum is over countable covers $\{U_i\}$ of $E$ with $\diam U_i<\delta$ $\forall i$.
\end{definition}

For the Teichm\"uller metric, there is a nontrivial $h$--dimensional Hausdorff measure (recalling that $h=6g-6$).
As we shall see, in order to understand average distances with respect to this measure, it will be necessary to compare with other measures, 
defined below, which are also natural to consider in their own right.

\subsection{Measures on Finsler manifolds}

The Teichm\"uller space  carries several  natural volume forms coming from its structure as a Finsler manifold. 
Let us discuss these general constructions first before returning to the case of $M=\T(S)$.
The treatment closely follows the survey by \'Alvarez and Thompson \cite{AT}.

Recall that a Finsler metric on an $n$--dimensional Finsler manifold $M$ is a continuous function $F\colon T(M)\to \R$ that restricts to a norm  on each tangent space 
$T_x(M)$. There is a  dual norm on each cotangent space $T_x^*(M)$. For a point $x\in M$, 
let $B_x\subset T_x(M)$ and $B_x^*\subset T_x^*(M)$ denote the unit balls for these two norms. A local coordinate system 
$(x_1,\dotsc,x_n)$ on $M$ induces a pair of isomorphisms
\begin{equation}\label{eqn:linearizeTangentSpaces}
\phi\colon T_x(M) \to \R^n \qquad\text{and}\qquad \psi\colon T_x^*(M)\to \R^n
\end{equation}
defined by writing vectors and covectors with respect to the dual bases $\{\partial_{x_1},\dotsc,\partial_{x_n}\}$ and $\{dx_1,\dotsc,dx_n\}$. 
By definition of the dual norm, the pairing $T_x(M) \times T_x^*(M) \to \R$ is sent to the standard inner product on $\R^n$ under these isomorphisms. 
In the local coordinate chart we may now define two functions
\begin{equation*}
f(x) = \frac{\varepsilon_n}{\lambda\left(\phi(B_x)\right)}\qquad\text{and}\qquad g(x) = \frac{\lambda\left(\psi(B_x^*)\right)}{\varepsilon_n},
\end{equation*}
where $\lambda$ is Lebesgue measure and $\varepsilon_n:=\lambda({\rm Ball}^n)$ is the Lebesgue measure of the standard unit ball in $\R^n$. While these functions clearly depend on the choice of coordinates $(x_1,\dots,x_n)$, one may easily check that the $n$--forms
\[f(x)\, dx_1\wedge\dotsb\wedge dx_n \qquad\text{and}\qquad g(x) \, dx_1\wedge\dotsb\wedge dx_n\]
are independent of the coordinate system and therefore define global volume forms on $M$. 
The former is called the \emph{Busemann volume} on the Finsler manifold  and the latter is the \emph{Holmes--Thompson volume}; 
see  \cite{AT} for more details.  
These both define measures on $M$.

A third measure to consider is the one induced by the canonical symplectic form $\omega$ on the cotangent bundle,
defined as follows. Consider local coordinates $(x_1,\dotsc,x_n)$ 
defined in a neighborhood $U\subset M$. The $1$--forms $dx_1,\dotsc, dx_n$ then give a trivialization of $T^*(M)$ over $U$, 
and we have a local coordinate system on $T^*(M)$ given by 
\begin{equation}\label{eqn:symplecticCoords}
(x_1,y_1,\dotsc,x_n,y_n)\mapsto \left((x_1,\dotsc,x_n),\sum_{i=1}^n y_i \, dx_i\right).
\end{equation}
In these coordinates the canonical symplectic form may be written simply as $\omega = \sum dx_i\wedge dy_i$.
Taking exterior powers then yields a volume form $\sympmeas=\omega^n/n!$  on $T^*(M)$.
By restricting to the unit disk bundle $T^{*,\le 1}(M)$ and pushing forward by the projection $\pi\colon T^*(M)\to M$, we obtain 
a {\em symplectic measure} $\n$ on $M$.

Finally, a Finsler metric on a smooth manifold $M^n$ induces a path metric $d$ in the usual way, and this in turn gives rise to a Hausdorff measure in any dimension. 

Recall that a centrally symmetric convex body $\Omega\subset \R^n$ determines a \emph{polar body} 
$\Omega^\circ\subset (\R^n)^*=\R^n$ via
\[\Omega^\circ := \{\xi\in\R^n \mid \xi\cdot v \leq 1 \;\forall v\in \Omega\}.\]
The \emph{Mahler volume} of $\Omega$ is then defined to be the product 
$\Mah(\Omega):=\lambda(\Omega)\sdot \lambda(\Omega^\circ)$ of the Lebesgue volumes of $\Omega$ and $\Omega^\circ$. 
For any centrally symmetric convex body $\Omega$, it is known that
\begin{equation}\label{eqn:MahlerIneqs}
\frac{\varepsilon_n^2}{n^{n/2}} \leq \Mah(\Omega) \leq   \varepsilon_n^2= \Mah({\rm Ball}^n).
\end{equation}
The first inequality was established by John \cite{J}, and the latter, which gives an equality if and only if the norm is Euclidean,  is known as the Blaschke--Santal\'o inequality \cite{Santalo}.

\begin{theorem}[Assembling facts on Finsler measures]\label{finsler-compare}
Suppose that $M^n$ is a continuous  Finsler manifold.  
Then
\begin{itemize}
\item the Busemann measure $\busemeas$ and  the $n$--dimensional Hausdorff  measure $\hausmeas$ are equal;
\item the Holmes--Thompson measure $\htmeas$ and the symplectic measure $\n$ are 
scalar multiples: $\htmeas=\frac{1}{\varepsilon_n}\n$;
\item $\htmeas \leq \busemeas \leq (n^{n/2})\, \htmeas$, with equality of measures if and only  if the metric is 
Riemannian.
\end{itemize}

\end{theorem}

Note that it is still possible for $\htmeas$ and $\busemeas$ to be scalar multiples of each other in the non-Riemannian case, 
for instance on a vector space with a Finsler norm.

\begin{proof}
The first statement was originally shown by Busemann in the 1940s in \cite{B} and is stated in modern language in 
\cite[Thm 3.23]{AT}.

The second statement is straightforward and we include a proof for completeness.
Working in the local coordinates and applying the Fubini theorem, we see that the Holmes--Thompson 
volume of a subset $E\subset M$ is given by:
\begin{align*}
\int_{E}g(x) \, dx_1\wedge\dotsb\wedge dx_n
&= \int_E \left(\int_{\psi(B_x^*)}\frac{1}{\varepsilon_n} d\lambda\right)dx_1\wedge\dotsb\wedge dx_n\\
&= \frac{1}{\varepsilon_n}\int_{\pi^{-1}(E)\cap  T^{*,\le 1}(M)} dy_1\wedge\dotsb\wedge dy_n \wedge dx_1\wedge \dotsb \wedge dx_n\\
&= \frac{1}{\varepsilon_n}  \, \n(E).
\end{align*}

For the third statement, recall that the measures are defined by 
\[\busemeas(E) =  \int_E f(x) \, dx_1\wedge\dotsb\wedge dx_n \quad\text{and} \quad
\htmeas(E) = \int_E g(x) \, dx_1\wedge\dotsb\wedge dx_n.\]

For each $x\in M$, the unit ball $B_x\subset T_x(M)$ is sent to a centrally symmetric convex body 
$\phi(B_x)\subset \R^n$ under the isomorphism $\phi$ defined in (\ref{eqn:linearizeTangentSpaces}). 
The polar body is exactly given by $\phi(B_x)^\circ = \psi(B_x^*)$. Therefore, the Mahler volume of $\phi(B_x)$ is 
\[\Mah(\phi(B_x)) = \lambda(\phi(B_x))\sdot \lambda(\psi(B_x^*)) = \varepsilon_n^2 \frac{g(x)}{f(x)}.\]
Combining with (\ref{eqn:MahlerIneqs}) now implies that $n^{-n/2}f(x) \leq g(x) \leq f(x)$ for all $x\in M$. We conclude that 
$\htmeas(E) \leq \busemeas(E)\leq  n^{n/2}\htmeas(E)$ for all $E\subset M$.
Finally, since Blaschke--Santal\'o can only give equality for a Euclidean norm, it follows that $\busemeas$ and $\htmeas$ 
can only be equal for a Riemannian metric.
\end{proof}

\subsection{Measures coming from quadratic differentials}

Recall that quadratic differential space $\QD(S)$ is naturally identified with the cotangent bundle $T^*(\T(S))$ of Teichm\"uller space, and that each quadratic differential $q\in \QD(S)$ has a norm $\norm{q}$ given by the area of the flat structure on $S$ induced by $q$. The unit disk bundle for this norm will be denoted by
\[\diskbund(S) = \{q\in\QD(S) : \norm{q}\leq 1\}.\]
Using this disk bundle, the natural symplectic measure $\sympmeas$ on $\QD(S)$ descends to a measure $\n$ on $\T(S)$ exactly as above.  
We note that $\omega$ and therefore $\sympmeas$ and $\n$ are invariant under the action of the mapping class group.

The space $\QD(S)$ also carries a natural $\Mod(S)$--invariant measure $\muhol$ that is defined in terms of holonomy coordinates and which we will refer to as {\em holonomy measure}; it is also sometimes called Masur--Veech measure in the literature (see \cite{M} for details). This measure has been studied extensively, for instance to establish ergodicity results for the geodesic flow.
The measure $\muhol$ is also related to the ``Thurston measure'' $\thumeas$ on the space of measured foliations $\MF$ 
induced by the piecewise-linear structure of $\MF$ \cite{FLP}. Indeed, as seen in \cite{M}, $\muhol$ is equal to the pullback of 
$\thumeas\times\thumeas$ under the $\Mod(S)$--invariant map $\QD(S)\to\MF\times\MF$ that sends a quadratic differential to its vertical and horizontal foliations.

Just as $\sympmeas$ induces $\n$, the holonomy measure 
$\muhol$ descends to a measure $\m$ on $\T(S)$. Explicitly, the $\m$--measure of a set $E\subset \T(S)$ is given by
\[\m(E):= \muhol\left(\pi^{-1}(E)\cap \diskbund(S)\right).\]
This measure $\m$ has been studied previously in \cite{ABEM} and \cite{EM}.

\begin{proposition}\cite[p.3746]{Mas}\label{prop:holVSsymp} There is a scalar $k>0$ such that $\sympmeas=k\sdot\muhol$.
\end{proposition}

We recall the outline of the argument here.
In \cite{Mas}, it was shown that the Teichm\"uller geodesic flow on $\QD(S)$ is a Hamiltonian flow for the function
\[H(q) = \frac{\norm{q}^2}{2}.\]
As such, the Teichm\"uller flow preserves the symplectic form $\omega$ and the corresponding measure $\sympmeas$. The measures 
$\sympmeas$ and $\muhol$ both descend to the quotient space $\QD(S)/\Mod(S)$; furthermore, the latter defines an ergodic measure for the 
Teichm\"uller flow on $\QD(S)/\Mod(S)$ \cite{M}. Since $\sympmeas$ is absolutely continuous with respect to $\muhol$, the proposition follows. 

We therefore also have $\n=k\m$, and combining Proposition~\ref{prop:holVSsymp} with Theorem~\ref{finsler-compare} we get:
\begin{corollary}
\label{haus-vs-hol} There are scalars $k_2>k_1>0$ such that 
$$k_1 \m \le \hausmeas \le k_2 \m.$$
\end{corollary}

\subsection{Visual measures}\label{vismeas}
The unit sphere subbundle of $\QD(S)$ will be denoted by
\[\QD^1(S)=\{q\in \QD(S) \colon \norm{q} = 1\}.\]
For each $x\in\T(S)$, the fiber $\QD^1(x)$ is identified with the ``space of directions'' at $x$, and the Teichm\"uller geodesic flow 
$\varphi_t\colon \QD(S)\to\QD(S)$ gives rise to a homeomorphism
$$\begin{array}{rccc}
\Psi_x\colon &\QD^1(x)\times(0,\infty) &\to& \T(S)\setminus\{x\}\\
&(q,r) &\mapsto& \pi(\varphi_r(q)) \ ,
\end{array}$$
which serves as ``polar coordinates'' centered at $x$. Furthermore, this conjugates $\varphi_t$ to a \emph{radial  flow} 
based at $x$ given by 
\[\hat{\varphi}_t(\pi(\varphi_r(q))):= \pi(\varphi_{r+t}(q)).\]
We will consider measures on $\T(S)$ that are compatible with these polar coordinates and with the radial  flow.

\begin{definition}[Visual measure]
Given any measure $\kappa_x$ on the unit sphere $\QD^1(x)\cong \mathbb{S}^{h-1}$, we define the corresponding \emph{visual measures} on $\Sph{r}{x}$ and $\T(S)$ as follows. Firstly, the visual measure $\vis_r(\kappa_x)$ on the sphere $\Sph{r}{x}$ of radius $r$ is just the push-forward of $e^{hr}\kappa_x$ under the homeomorphism $\QD^1(x)\times\{r\} \cong \Sph{r}{x}$. Integrating these over $(0,\infty)$ then gives a visual measure on $\T(S)$ defined by 
\[\vis(\kappa_x)(E) := \int_{(q,r)\in E\subset \QD^1(S)\times (0,\infty)}e^{hr}d\kappa_x(q)d\lambda(r).\]
Said differently, $\vis(\kappa_x)$ is equal to the push-forward of $\kappa_x\times \lambda_0$ under the homeomorphism $\Psi_x$, where $\lambda_0$ is the weighted Lebesgue measure on $(0,\infty)$ given by  $\lambda_0([a,b]) = \int_a^b e^{hr}d\lambda(r) = (e^{hb}-e^{ha})/h$. (We have scaled things in this way so that the visual measure of the ball of radius $R$ grows like $e^{hR}$.) 
\end{definition}

The essential feature of visual measures is that they enjoy the following ``normalized invariance'' under the radial  flow: For any $t\geq 0$ and measurable  $E\subset \Sph{r}{x}$ we have
\[\frac{\vis_{r+t}(\kappa_x)(\hat{\varphi}_t(E))}{\vis_{r+t}(\kappa_x)(\Sph{r+t}{x})} = \frac{\vis_r(\kappa_x)(E)}{\vis_r(\kappa_x)(\Sph{r}{x})}.\]
The same invariance holds for $\vis(\kappa_x)$ when we normalize with respect to annular shells $\Ball{b}{x}\setminus\Ball{a}{x}$ instead of spheres.

There are two visual measures that specifically interest us. Firstly, the normed vector space $\QD(x)$ carries a unique translation-invariant measure $\nu_x$ normalized so that $\nu_x(B_x^*) = 1$; recall that the unit ball $B_x^*$ is just the intersection $\QD^{\leq 1}(S)\cap \QD(x)$. 
This induces a measure (also denoted $\nu_x$) on the unit sphere $\QD^1(x)$ via the usual method of coning off: $\nu_x(E) := \nu_x\left([0,1]\times E\right)$ for $E\subset \QD^1(x)$.

Secondly, since $\QD(S)$ has the structure of a fiber bundle over $\T(S)$, we can define a conditional measure $s_x$ on $\QD(x)$ by disintegration from $\muhol$. 
More precisely, $s_x$ is the unique measure on $\QD(x)$ such that the $\muhol$--measure of $E\subset \QD(S)$ is given by
\[\muhol(E) = \int_{\T(S)} s_x(E\cap\QD(x))\,d\m(x).\]
Via the process of coning off, we again think of $s_x$ as a measure on $\QD^1(x)$.

The space $\QD(S)$ of quadratic differentials is a complex vector bundle; as such, there is a natural circle action 
$S^1\curvearrowright \QD(S)$ that preserves each fiber $\QD(x)$ and unit sphere $\QD^1(x)$. We say that a visual measure 
$\vis(\kappa_x)$ is \emph{rotation-invariant} if the corresponding measure $\kappa_x$ on $\QD^1(x)$ is invariant under this 
action of $S^1$. The visual measure $\vis(\nu_x)$ is rotation-invariant because $S^1$ preserves the unit ball $B_x$. Similarly, 
$\vis(s_x)$ is rotation-invariant because $S^1$ preserves $\muhol$.

\subsection{Summary}\label{meas-comp-sum}

The measures on $\T(S)$ considered above are $\n$ and $\m$ (induced by the symplectic and 
holonomy measures on $\QD(S)$, respectively, via the covering map), Hausdorff measure $\hausmeas$, the visual measures
$\vis(\kappa_x)$ created by radially flowing measures on the sphere of directions $\QD^1(x)$, and the measures
$\busemeas$ and $\htmeas$ coming from the Finsler structure.

We found that $\n$, $\m$, and  $\htmeas$ 
are scalar multiples of each other, Hausdorff measure and Busemann measure coincide, 
and all five of these are mutually comparable in the sense of being bounded above and below by scalar multiples of each
other. In the following section we will establish results about the structure of generic geodesic rays with respect to  these measures
and the visual measures.

%%%%

%%%%%%%%%
\section{Thickness statistics for geodesic rays}
\label{sec:separationAndThickness}
\label{sec:volumeAndDistEstimates}

In \S\ref{sec:thick_geodesics} we studied the behavior of Teichm\"uller geodesics that spend a definite fraction of their time in some thick part $\T_\epsilon$. In this section we will show that most Teichm\"uller geodesics in fact satisfy this property. Therefore the tools developed in \S\ref{sec:thick_geodesics} apply generically and we may use them in studying averaging questions such as Theorems \ref{thm:balls}, \ref{thm:spheres}, and \ref{T:expected_thinness}.

In all of what follows, if $x$ is a fixed basepoint and $y\in \Ball{r}{x}$ is a point in the ball centered at $x$, then we will write $y_t$ to denote the time--$t$ point on the geodesic ray based at $x$ and traveling through $y$.

\subsection{Volume estimates}
We begin by recalling some estimates on the volume of Teichm\"uller balls and using these to reduce to the case of annular shells.

Athreya, Bufetov, Eskin, and Mirzakhani \cite{ABEM} have found the following asymptotic estimate for the $\m$--volume of a ball of radius $r$.

\begin{theorem}[Volume asymptotics {\cite[Theorem 1.3]{ABEM}}]\label{thm:volumeBalls}
There is a (bounded) function $f\colon\T(S) \to (0,\infty)$ such that for each $x\in \T(S)$
\[\lim_{r\to \infty} \frac{\m(\Ball{r}{x}))}{e^{hr}} = f(x).\]
\end{theorem}

\begin{corollary}[Definite exponential growth]
\label{cor:expGrowthBalls}
Let $\mu$ denote Hausdorff measure $\hausmeas$, holonomy measure $\m$, 
or any visual measure $\mu_x=\vis(\kappa_x)$. 
For each $x\in \T(S)$, there exist constants $C_1\leq C_2$ such that for all sufficiently large $r$ (depending on $x$) we have
\[C_1 e^{hr} \leq \mu(\Ball{r}{x})\leq C_2 e^{hr}.\]
\end{corollary}
\begin{proof}
This is built into the definition of the visual measure $\vis(\kappa_x)$. For the holonomy measure $\m$, this follows from Theorem~\ref{thm:volumeBalls} above. The same estimate then holds for $\hausmeas$ by Proposition~\ref{prop:holVSsymp} and Corollary~\ref{haus-vs-hol}.
\end{proof}

Of course, this holds for the other measures discussed in this paper as well by the comparisons in the 
last section.

For any $r > k > 0$, let $\Ann{r}{x}{k} = \Ball{r}{x} \setminus \Ball{r-k}{x}$ denote the annular shell between radii $r$ and $r-k$. The fact that the volume of a ball grows exponentially in the radius means that we can focus our attention on annuli rather than on balls.

\begin{lemma}[Reduction to annuli]
\label{lem:annuliReduction}
Fix $x\in \T(S)$ and let $\mu$ be any measure with definite exponential growth (i.e., satisfying the conclusion of Corollary~\ref{cor:expGrowthBalls}). 
Suppose that for all $k > 0$ we have
\[\lim_{r\to \infty} \frac{1}{r} \frac{1}{\mu(\Ann{r}{x}{k})^2}\int_{\Ann{r}{x}{k}\times\Ann{r}{x}{k}}d_\T(y,z) \ d\mu(y)d\mu(z) = 2.\]
Then the same holds when $\Ann{r}{x}{k}$ is replaced  by $\Ball{r}{x}$.
\end{lemma}
\begin{proof}
Let $C_1, C_2$ be as in Corollary~\ref{cor:expGrowthBalls} above. For each $k$ sufficiently large (satisfying $\frac{C_2}{C_1}e^{-hk} < 1$) and all sufficiently large $r$ we have
\begin{align*}
2 &\geq
\frac{1}{r}\frac{1}{\mu(\Ball{r}{x})^2}\int_{\Ball{r}{x}\times\Ball{r}{x} }d_\T(y,z) \ d\mu(y)d\mu(z)\\
&\geq \left(\frac{\mu(\Ball{r}{x})}{\mu(\Ball{r}{x})} - \frac{\mu(\Ball{r-k}{x})}{\mu(\Ball{r}{x})}\right)^2 \frac{1}{r}\frac{1}{\mu(\Ann{r}{x}{k})^2}\int_{\Ann{r}{x}{k}\times\Ann{r}{x}{k}}d_\T(y,z) \ d\mu(y)d\mu(z)\\
&\geq \left(1 - \frac{C_2}{C_1}e^{-hk}\right)^2 \frac{1}{r}\frac{1}{\mu(\Ann{r}{x}{k})^2}\int_{\Ann{r}{x}{k}\times\Ann{r}{x}{k}}d_\T(y,z) \ d\mu(y)d\mu(z).
\end{align*}
The claim now follows since, by assumption, the latter becomes arbitrarily close to $2$ when $r$ and $k$ are sufficiently large.
\end{proof}

\subsection{The thickness property}\label{s:thick-stat}
Recall from \S\ref{sec:thick_geodesics} that the thick-stat of a nondegenerate geodesic $[x,y]\subset \T(S)$ is defined by
\[\thickfrac_{\epsilon}[x,y]= \frac{\bigl| \{ 0 \le s \le d_\T(x,y) : y_s \in \T_{\epsilon}\}\bigr|}{d_\T(x,y)},\]
where $y_s$ denotes the time--$s$ point on the geodesic ray from $x$ through $y$.
 The goal of this section is to show that, for all measures of interest, the thick-stat $\thickfrac_{\epsilon}[x,y]$ is uniformly bounded below for most $y\in \Ann{r}{x}{k}$.
More precisely, we will show that these measures satisfy the following property.

\begin{definition}[Property P1]
 \label{thickness}
We say a measure $\mu$ on $\T(S)$ has the {\em thickness property} (P1) if for all $0<\theta,\sigma<1$, there exists $\epsilon > 0$ such that 
\[\lim_{r\to\infty} \frac{\mu \left( \bigl\{y\in \Ann rxk : \thickfrac_{\epsilon}[x,y_t]\ge \theta ~\hbox{\rm for all}~ \sigma r \le t\le r \bigr\} \right)}{\mu(\Ann rxk)} = 1\] 
holds for all $x\in \T(S)$ and $k>0$.
\end{definition}

We first observe that, for visual measures, the thickness property follows from the  ergodicity of the Teichm\"uller geodesic flow 
$\varphi_t$. 

\begin{proposition}[Thickness statistics for visual measures]
\label{prop:thick-vis}
Let $\kappa_x$ denote either of the visual measures $s_x$ or $\vismeas$ on $\QD^1(x)$. For all $0 < \theta < 1$ there exists $\epsilon > 0$ such that for all $x\in \T(S)$ we have
\[\lim_{R_0\to \infty} \kappa_x \left(\bigl\{q\in \QD^1(x) : \thickfrac_{\epsilon}[x,\pi(\varphi_r(q))]\ge \theta ~\hbox{\rm for all}~  r>R_0  \bigr\} \right) = 1.\]
\end{proposition}
\begin{proof}
Choose $\epsilon>0$ sufficiently small so that the proportion of the $\m$--volume of moduli space that is $\epsilon$--thick is larger than $\theta$; that is, so that
\[\m\bigl(\T_\epsilon/\Mod(S)\bigr) > \theta\cdot \m\bigl(\T(S)/\Mod(S)\bigr).\]
By the ergodicity of the geodesic flow \cite{M}, it follows that the geodesic ray determined by $\muhol$--almost every $q\in \QD^1(S)$ spends more than proportion $\theta$  of its time   
in $\T_{\epsilon}$, asymptotically. 
The vertical foliation of each such $q$ is uniquely ergodic \cite{M2}.  If 
 two quadratic differentials have the same vertical uniquely ergodic measured foliation then they are 
 forwards asymptotic \cite{M1}.
 We conclude that almost every measured foliation $F\in \MF$ 
(with respect to Thurston measure $\thumeas$)  has the property that any quadratic differential $q$ with 
vertical foliation $F$ satisfies $\lim_{r\to\infty}\thickfrac_\epsilon[\pi(q),\pi(\varphi_r(q))]\geq\theta$.

The map  $\qd(x)\to \MF$ which assigns to $q$ its vertical foliation is a smooth map off the 
multiple zero locus, so it is smooth on  a set of full measure.  
Thus it is absolutely continuous with respect to the measures $\kappa_x$ and $\thumeas$.
Thus the property of asymptotically spending at least proportion $\theta$  of the time in $\T_\epsilon$ holds for $\kappa_x$--almost every $q\in \qd(x)$. This means that for each $x\in\T(S)$ the quantity
\[\kappa_x \left(\bigl\{q\in \QD^1(x) : \thickfrac_{\epsilon}[x,\pi(\varphi_r(q))]\ge \theta ~\hbox{\rm for all}~  r>R_0  \bigr\} \right)\]
increases to $1$ as the threshold $R_0$ tends to infinity.
\end{proof}

Since visual measures on $\T(S)$ are obtained by integrating the above measures on spheres, we immediately obtain the thickness property for visual measures.

\begin{corollary}\label{cor:thickness_for_vis}
The thickness property {\rm (P1)} holds for  $\vis(s_x)$ and $\vis(\vismeas)$.
\end{corollary}

%%%%
\subsection{Random walks}
\label{sec:random_walks}

We next verify (P1) for $\m$, which is considerably more involved. 
The proof uses ideas of Eskin and Mirzakhani on discretizing geodesics into sample paths of a random walk.
We begin the setup by combining some results on the volume of balls from Athreya--Bufetov--Eskin--Mirzakhani \cite{ABEM} and Eskin--Mirzakhani \cite{EM}.

\begin{lemma}[{Volume of balls \cite[Theorem 1.2]{ABEM}, \cite[Lemma 3.1]{EM}}]
\label{lem:ballVolumes}
There exists a constant $\ccc > 0$ (depending only on the topology of $S$) such that $\m(\Ball{\ccc}{y})\emul 1$ for all $y\in \T(S)$. Additionally, given $\epsilon >0$ we have $\m(\Ball{r}{y})\lmul_{\epsilon} e^{hr}$ for all $y\in \T_\epsilon$. 
\end{lemma}
\begin{proof}
The first claim is exactly Lemma 3.1 of \cite{EM}. For the second claim, choose a point $x\in \T_{\epsilon}$. Note that our choice of $x$ depends only on $\epsilon$. By the volume asymptotics (Theorem~\ref{thm:volumeBalls}), there exists constant $R_0$ such that
$$\m(\Ball{r}{x})\lmul_{\epsilon} e^{hr}$$
for all $r\geq R_0$. Furthermore, by increasing $R_0$ if necessary, we may assume that the $\Mod(S)$--translates of $\Ball{R_0}{x}$ cover $\T_{\epsilon}$. It follows that for any $y\in \T_{\epsilon}$ and any $r\geq 0$ we have
\[\m(\Ball{r}{y}) \leq \m(\Ball{r+R_0}{x'}) = \m(\Ball{r+R_0}{x}) \lmul_\epsilon e^{hR_0}e^{hr}\]
for some $\Mod(S)$--translate $x'$ of $x$. This establishes the second claim.
\end{proof}

To define a random walk on $\T(S)$ with basepoint $x$, first choose a net $\N$ of points in $\T(S)$, choosing so that $x\in\N$ and such that the net points are $\ccc$--separated and $(2\ccc)$--dense (i.e., the distances between net points are at least $\ccc$ but the $(2\ccc)$--balls about net points cover \Teich space).  Here $\ccc$ is the constant from Lemma~\ref{lem:ballVolumes}, depending only on the topology of $S$.

Given a parameter  $\tau$, 
a {\em sample path of length $s$} (starting at $x$) is a
map
\[\lambda\colon \{0,\ldots, \floor{s/\tau}\}\to \N\]
such that $\lambda_0=x$ and for each index, $d_\T(\lambda_k,\lambda_{k+1})\le \tau.$
Let $\P_\tau^x(s)$ be the set of sample paths $\lambda$ starting at $x$ of length at most $s$, and let $\P_\tau^x$ be the set of all sample paths of any length. By (36) of \cite{EM},  for any 
$\delta>0$ and sufficiently large $\tau$ (depending on $\delta$) we have, $$|\P_\tau^x(s)|\leq e^{s(h+\delta)}$$ 
for all $s\geq 0$. (Note that the constant $C_2$ in (36) of \cite{EM}, coming from \cite[Proposition 4.5]{EM}, can be taken to equal 1.)

Now given $\tau$ we  define a map $F_\tau\colon \T(S) \to \P_\tau^x$ which takes a point $y$ and ``discretizes'' the geodesic $[x,y]$ to a sample path. For any  $[x,y]$ we mark off points along the geodesic starting at $x$ and spaced by time $\tau-2\ccc$. For each such point along the geodesic we choose a nearest point in $\N$.  This is the sample path  $F_\tau(y)$  associated to $y$.
For any $\epsilon_1>0$, we can choose  $\tau$  sufficiently large so that the image under $F_\tau$ of the ball $\Ball rx$ is contained in $\P_\tau^x(r(1+\epsilon_1))$. Furthermore, after increasing $\tau$ if necessary, the above estimate on the cardinality of $\P_\tau^x(s)$ shows that we additionally have 
\[\abs{\P_\tau^x(r(1+\epsilon_1))}\leq e^{hr(1+2\epsilon_1)}\]
for all $r$. We are now ready to show that the measure of points determining a ray with a too-large thick-stat decays exponentially in $r$.

\begin{theorem}[Thickness estimate for holonomy measure]
\label{thm:thick-hol}
For all
$0 < \theta,\sigma < 1$, $x\in\T(S)$, and $k>0$, there exist $\epsilon=\epsilon(\theta,\sigma)  >0$ and $\alpha>0$ such that 
 $$\frac{\m \left( \bigl\{y\in \Ann rxk : \thickfrac_{\epsilon  }[x,y_t]< \theta ~\hbox{\rm for some}~ t\in[\sigma r,r] \bigr\}   \right)}{\m(\Ann rxk)}  
 < e^{-\alpha r}$$ 
for all sufficiently large $r$, 
where $y_t$ is the time--$t$ point on the geodesic ray from $x$  through $y$.
\end{theorem}

\begin{proof}
We start with a geodesic $[x,y]$ for $y\in \Ann rxk$. 
As in the above discussion, for  any $\epsilon_1<1$, we can choose $\tau\geq \frac{4\ccc}{\epsilon_1}$  so that the geodesic determines a sample path $\lambda$ in $\P_\tau^x(r(1+\epsilon_1))$. Let $F_\tau\colon \Ann rxk\to P_\tau$ be the map carrying $y$ to the sample path associated to $[x,y]$.
In Appendix~\ref{sec:thickness_for_random_walks} we give a self-contained  proof of the following statement (Theorem~\ref{thm:percentage}) which is indicated in the proof of Theorem 5.1 of \cite{EM}:
given $\theta$, there exists $\delta'>0$ such that for all large $\tau$ there is an $\epsilon'$ so that for each $\tau \leq t \leq r$, the estimate
\[\frac{1}{\lfloor t/\tau \rfloor}\left|\bigl\{1\leq i\leq \lfloor t/\tau \rfloor: \lambda_i\in \T_{\epsilon'}\bigr\}\right|\geq \theta\]
holds for all but at most
$$e^{-\delta' t}e^{hr(1+2\epsilon_1)}$$
of the sample paths $\lambda\in \P_\tau^x(r(1+\epsilon_1))$. This is the crucial ingredient that lets us control our thickness statistic. For the given $\sigma$, we now choose $\epsilon_1$ small enough (forcing $\tau$ to be large) so that $\kappa:=\frac{\delta'\sigma}h-2\epsilon_1>0$. In particular, note that $\tau$ and $\delta'$ depend only on $\theta$ and $\sigma$. If we let $\Omega\subset \P_\tau^x(r(1+\epsilon_1))$ denote the union of these exceptional sample paths corresponding to any $\sigma r \leq t\leq r$, it follows that
\[\abs{\Omega} \leq \sum_{k=0}^\infty e^{-\delta'(\sigma r + k\tau)} e^{hr(1+2\epsilon_1)}\lmul_{\theta,\sigma} e^{-\delta'\sigma r}e^{hr(1+2\epsilon_1)} = e^{hr(1-\kappa)}.\] 

We know that the ball of radius $\ccc$ centered at any point has $\m$--measure $O(1)$
by Lemma~\ref{lem:ballVolumes}. 
Consequently, when $r$ is sufficiently large we have
\[\m(F_\tau^{-1}(\Omega))\lmul|\Omega|\lmul_{\theta,\sigma} e^{hr}e^{-\kappa h r}  \leq\m(\Ann rxk) \cdot e^{-\alpha r},\]
for a suitable choice of $\alpha$.
We conclude that for all   $y\in \Ann rxk$ except for a set of at most this  measure, the  sample path associated to the geodesic from $x$ has to $y$ has $\thickfrac_{\epsilon'}\ge \theta$ for ending times $t\geq \sigma r$. 
Now every point on the geodesic is within distance $\tau$ of a point on the sample path. Let $\epsilon=\epsilon' e^{-\tau}$. A point at distance at most $\tau$ from a point in $\T_{\epsilon'}$ lies in $\T_{\epsilon}$. This concludes the proof. 
\end{proof}

This says that, except for set of endpoints $y$ of exponentially small measure, geodesics have the property that they eventually have spent a definite fraction of their time in the thick part. In particular, this implies the following.

\begin{corollary}\label{cor:thickness_for_hol}
The measures $\m$ and $\hausmeas$ satisfy the thickness property {\rm (P1)}.
\end{corollary}

We will again use random walks to now show that a typical geodesic has a long interval where it stays in the thick part.  
Specifically we say that a geodesic segment $[x,y]$ contains an $\epsilon$--{\em thick interval} $I$ if 
there is a geodesic subsegment $I$  along  $[x,y]$ such that $I\subset \T_\epsilon$.

\begin{theorem}[Thick intervals]
\label{thm:interval}
For all $0<\sigma<1$, $M>0$, and sufficiently small $\epsilon$, 
there exists $\beta>0$ such that for all sufficiently large $r$, 
 $$\frac{\m \left( \bigl\{y\in \Ann rxk : [y_{\sigma r},y_{2\sigma r}]\text{
contains no  $\epsilon$--thick interval of length } M\bigr\}\right)}{ \m(\Ann rxk)}<
 e^{-\beta r},$$
where $y_t$ is the time--$t$ point on the geodesic ray from $x$ through $y$.
\end{theorem}

\begin{proof}
For small  $\epsilon_1$ define $\tau$ as in  Theorem~\ref{thm:thick-hol}.  Fix some $0<\theta<1$ by Theorem~\ref{thm:thick-hol}, for sufficiently small $\epsilon'$,  except for an exponentially small set of paths the interval $[y_{\sigma r},y_{2\sigma r}]$ contains a proportion $\theta$ of $\epsilon'$--thick points.  
There exists $\rho>0$ such that  for any net point $y\in \T_{\epsilon'}$, the 
probability that the next step in the random walk starting at $y$ remains in  
$\T_{\epsilon'}$ is at least $\rho$.

Let 
$\kappa:=1-\rho^{M/\tau}<1$.  Then given a point $y\in \T_{\epsilon'}$,  the probability that in the next $\lfloor M/\tau \rfloor$ steps in the 
random walk (that is, the sample path of length $M$) starting at $y$ 
at least one of the points is $\epsilon'$--thin is at most $\kappa$.

Consider a sample path  of length $L=\sigma r$ (so having $\lfloor L/\tau \rfloor$ points)
for which at least $\theta$ proportion of its points are  in $\T_{\epsilon'}$.  Divide the path  into subpaths with  $\lfloor M/\tau \rfloor$ points, 
called {\em pieces},
so that the number of pieces is additively close to $L/M$ if $M\gg \tau$.  
Let $J_1$ be the collection of odd-index pieces and $J_2$ the collection of even-index pieces. 
Suppose without loss of generality that $J_1$ contains at least as many pieces with an $\epsilon'$--thick point as $J_2$ does.
The number of  $J_1$ pieces containing a  point in $\T_{\epsilon'}$ is then at least   $N:=\frac{L\theta}{2 M}$.
The probability that the sample path of length $M$  starting at an $\epsilon'$ thick point enters the $\epsilon'$--thin part is at most $\kappa$, as we have seen. 
If two thick starting points are in different $J_1$ pieces, then these events are independent,  because between the two points  is a $J_2$ piece of length $M$.    
Thus the probability that the random path does not have any pieces  of length $M$ that lie entirely in $\T_{\epsilon'}$ is at most  $\kappa^N$.

Now as in the discussion in the previous theorem, we have the map $F_\tau\colon\Ann rxk \to\P_\tau^x(r(1+\epsilon_1))$.  Let $y\in \Ann rxk$.
We consider the segment $[y_{\sigma r},y_{2\sigma r}]\subset [x,y]$.  
Its image under $F_\tau$ is a path of length $L=\sigma r$. 
The probability that this random path fails to have the desired segment of length $M$ in the $\epsilon'$--thick part is at most 
$\kappa^{\textstyle\nicefrac{\sigma r\theta}{2M}}$,
by the statement at the end of the last paragraph.  Since $\kappa$ is a fixed number smaller than  $1$, 
we can find an upper bound for this proportion of the form  $e^{-\beta r}$ for some $\beta>0$.  As we saw in the proof of the last theorem,
since this property holds for an exponentially small proportion of sample paths,  the corresponding property holds for an 
exponentially small proportion of endpoints $y\in \Ann rxk$.

Let $\epsilon=e^{-\tau}\epsilon'$.
Note again that if a point on the  random path is $\epsilon'$--thick, then the corresponding point on the geodesic $[x,y]$ is  $\epsilon$--thick. 
We conclude, as in the last theorem, that the  measure of the set of points $y\in \Ann rxk$ such that the geodesic $[x,y]$ does not have an $\epsilon$--thick interval of length $M$ in $[\sigma r,2\sigma r]$ is at most $e^{-\beta r}$.
\end{proof}

%%%%%%%%%%%%%%%%%%%%%%%
\section{Separation statistics for pairs of rays}
\label{s:fellow-trav-stat}

We need one final ingredient before proving Theorems~\ref{thm:balls} and \ref{thm:spheres}. Namely, in order to apply Theorem~\ref{thm:partiallyThickIntervalsAreClose} to show that the geodesic $[y,z]$ connecting two generic points $y,z\in \Ball rx$ must ``dip back'' towards $x$, we must first know that $[x,y]$ and $[x,z]$ become $C$--separated, where $C$ is the constant from Theorem~\ref{thm:partiallyThickIntervalsAreClose}. Thus we need an estimate for the probability that two geodesic rays based at $x$ fellow-travel past a given radius. The appropriate sort of control is ensured by the following property.

\begin{definition}[Property P2]
We say a measure $\mu$ on $\T(S)$ satisfies the 
{\em separation property} (P2) if for all $M_0,k>0$, $0<\sigma < 1$, 
and  $x\in \T(S)$, we have
\[\lim_{r\to\infty} \frac{\mu\times\mu \left( \bigl\{(y,z) \in \Ann rxk \times \Ann rxk : d_\T(y_t, z_t) \geq M_0 ~\text{for all}~t\in [\sigma r,r]\bigr\}\right)}
{\mu\times \mu \bigl( \Ann rxk \times \Ann rxk \bigr)} = 1.\]
\end{definition}

We will also consider the following time-specific version this separation property:

\begin{definition}[Property P3]
We say that a measure $\mu$ on $\T(S)$ satisfies the \emph{strong separation property} (P3), or has exponential decay of fellow travelers, 
if for all $x\in \T(S)$, $M_0,k > 0$, $0 < \sigma <1$, there exist $\alpha,R_0 > 0$
such that
\[\frac{\mu\times\mu \left( \bigl\{(y,z)\in \Ann rxk\times \Ann rxk :  d_\T(y_t,z_t) < M_0  \bigr\}\right)}{\mu\times \mu \bigl( \Ann rxk \times \Ann rxk \bigr)}
 \leq e^{-\alpha t},\]
whenever $R_0 \leq \sigma r \leq t \leq r$.
\end{definition}

\begin{theorem}[Strong separation for visual measures]\label{thm:exp-decay-vis}
All rotation-invariant visual measures $\mu_x = \vis(\kappa_x)$ on $\T(S)$, and in particular  $\vis(\vismeas)$ and $\vis(s_x)$,   
have the strong separation property {\rm (P3)}.
\end{theorem}
\begin{proof}
Choose $\sigma r \leq t \leq r$, fix a point $y\in \Ann{r}{x}{k}$, and let $E = \{z\in \Ann rxk : d_\T(y_t,z_t) < M_0\}$.  Looking instead in the sphere $\Sph{t}{x}$, we have the set $E_t=\{z\in \Sph{t}{x} : d_\T(y_t,z) < M_0\}$. Notice that, by definition,
\[E = \bigcup_{s\in [r-t-k,r-t]}\hat{\varphi}_{s}(E_t).\]
(Recall that $\hat{\varphi}_s$ denotes the radial geodesic flow based at $x$.) Therefore, by the normalized invariance, we have
\begin{align*}
\mu(E) 
&= \int_{r-k}^r \vis_s(\kappa_x)(\hat{\varphi}_{s-t}(E_t)) \,d\lambda(s)\\
&= \int_{r-k}^r \vis_t(\kappa_x)(E_t) \frac{\vis_s(\kappa_x)(\Sph{s}{x})}{\vis_t(\kappa_x)(\Sph{t}{x})}\,d\lambda(s)\\
&= \frac{\kappa_x(E_t)}{\kappa_x(\QD^1(x))} \mu_x(\Ann{r}{x}k),
\end{align*}
where, in the last line, we have identified $E_t$ with its image in $\QD^1(x)\cong \Sph{t}{x}$. 

It remains to find $R_0$ (independent of $y$) such that $\kappa_x(E_t)/\kappa_x(\QD^1(x))\lmul e^{-t}$ when 
$t\geq R_0$. Recall that $S^1$ acts freely on 
$\QD^1(x)$ by rotations. Choosing orbit representatives, we may realize $\QD^1(x)$ as a setwise product 
$(\QD^1(x)/S^1)\times S^1$. The measure $\kappa_x$ pushes forward to a measure on $\QD^1(x)/S^1$. 
By disintegration, we then obtain a measure on each fiber $S^1$ which, by the rotation-invariance of $\kappa_x$, 
must agree with Lebesgue measure up to a scalar. For any two points $z,z'\in E_t$, the triangle inequality gives 
$d_\T(z,z')\leq 2 M_0$. Now suppose that $z$ and $z'$ lie in the same 
Teichm\"uller disk, meaning that the unit quadratic differentials associated to the geodesics $[x,z]$ and $[x,z']$ lie in the same $S^1$--orbit. Each Teichm\"uller disk is an isometrically embedded copy of the hyperbolic plane. Thus, when $t$ is large compared to $M_0$, hyperbolic geometry implies that the fraction of each $S^1$--orbit contained in $E_t$ is  $\lmul e^{-t}$. Using the product structure and integrating over the $\QD^1(x)/S^1$ factor, Fubini's theorem then implies that $\kappa_x(E_t)/\kappa_x(\qd^1(x)) \lmul e^{-t}$ as well.
\end{proof}

\begin{theorem}[Strong separation for holonomy measure]
\label{thm:exp-decay-hol}
The measure $\m$ satisfies the strong separation property {\rm (P3)}.
\end{theorem}
\begin{proof}For a given $r,t$, let
\[D_{r,t} = \{(y,z) : d_\T(y_t,z_t) < M_0\} \subset  \Ann{r}{x}{k}\times\Ann{r}{x}{k}    \]
denote the the set in question. Set $\theta = \nicefrac{1}{2}$ and choose $\epsilon$ and $\alpha > 0$ as in Theorem~\ref{thm:thick-hol} so that the $\m$--measure of the set
\[E_r = \bigl\{y : \thickfrac_{\epsilon}[x,y_t]< \theta ~\hbox{\rm for some}~ t\in[\sigma r,r]\bigr\} \subset \Ann{r}{x}{k} \]
is at most $\m(\Ann rxk)e^{-\alpha r}$ for all large $r$. We may now write $D_{r,t} = D'_{r,t}\cup D''_{r,t}$, where
\[D'_{r,t} = \{(y,z)\in D_{r,t} : y\in E_{r}\}\quad\text{and}\quad D''_{r,t}=\{(y,z)\in D_{r,t} : y\notin E_r\}.\]
By the above, we have that \begin{equation}
\label{measure1}
\frac{\m\times \m(D'_{r,t})}{\m(\Ann{r}{x}{k})^2} \leq \frac{\m(E_r)\m(\Ann rxk)}{\m(\Ann rxk)^2} \leq e^{-\alpha t}
\end{equation}
 for all large $r$ and all $t\leq r$.

Choose any point $y\in \Ann{r}{x}{k}\setminus E_r$ and fix some $t\in [\sigma r, r]$. Then any point $z\in \Ann{r}{x}{k}$ satisfying 
$d_\T(y_t,z_t) < M_0$  must be contained in the ball of radius $r-t+M_0$ about $y_t$. Since $t\geq \sigma r$ and 
$y\notin E_r$; we know that $\thickfrac_{\epsilon}[x,y_t]\geq \theta$. 
Choosing any $0 < \delta < \theta$, there must exist a time $t'\in [\delta t,t]$ for which $y_{t'}\in \T_{\epsilon}$. Furthermore, each such $z$ lies within the ball of radius 
$(1-\delta)t + M_0 + r-t$ about $y_{t'}$. Applying Lemma \ref{lem:ballVolumes}, we see that
\begin{equation}
\label{measure2}
\frac{\m\times\m(D''_{r,t})}{\m(\Ann{r}{x}{k})^2} \leq \frac{\m(\Ball{r-\delta t + M_0}{y_{t'}})}{\m(\Ann rxk)} \lmul_\epsilon \frac{e^{hM_0}e^{hr}e^{-h\delta t}}{\m(\Ann rxk)}.
\end{equation}
The Theorem  now follows from (\ref{measure1}), (\ref{measure2}) and the fact that, by Corollary~\ref{cor:expGrowthBalls}, we have  $\m(\Ann{r}{x}{k})\gmul_{x,k} e^{hr}$ for all large $r$.
\end{proof}

Finally, we see that for any measure enjoying exponential decay of fellow travelers, most pairs of geodesic rays are in fact \emph{never} near each other beyond some threshold.

\begin{proposition}\label{lem:sphereEstimate} 
The strong separation property {\rm (P3)} implies the separation property {\rm (P2)}.
\end{proposition}
\begin{proof}
If $(y,z)\in \Ann rxk\times \Ann rxk$ does not lie in the set
\[E_r = \bigl\{(y,z) \in \Ann rxk \times \Ann rxk : d_\T(y_t, z_t) \geq M_0 ~\text{for all}~t\in [\sigma r,r]\bigr\},\]
then there is some $n\in \mathbb{N}$, $n\leq (1-\sigma)r$, such that 
$d_\T(y_{\sigma r+n},z_{\sigma r+n}) < M_0+2$. 
Thus all such points are contained in the union of exceptional sets  
corresponding to the radii $\sigma r, \sigma r+1, \dotsc, \sigma r +\floor{r-\sigma r}$. 
Using the exponential bound provided by property (P3), we see that for large $r$ the complement of $E_r$ has measure at most
\begin{align*}
\mu(\Ann{r}{x}{k})^2&\left(e^{-\alpha \sigma r} + e^{-\alpha(\sigma r + 1)} + \dotsb + e^{-\alpha(\sigma r + \floor{r-\sigma r}}\right)\\
&\leq\mu(\Ann{r}{x}{k})^2\left(\frac{1}{1-e^{-\alpha}}\right)e^{-\alpha\sigma r}.\qedhere
\end{align*}
\end{proof}

\begin{corollary}\label{cor:sep_property}
The Hausdorff measure $\hausmeas$, holonomy measure $\m$, and the visual measures $\vis(\vismeas)$ and $\vis(s_x)$ all satisfy property {\rm (P2)}.
\end{corollary}

Thus we can conclude that after throwing out a subset of $\Ann rxk\times\Ann rxk$ of an arbitrarily small proportional measure, all pairs of geodesics stay separated by an arbitrarily chosen distance in Teichm\"uller space after a threshold time $\sigma r$ has elapsed.

\section{Statistical hyperbolicity}
\label{sec:hyperbolicity}

We can now assemble our results to prove 
Theorems~\ref{thm:balls} and \ref{thm:spheres}.

\begin{theorem}[Annulus version of statistical hyperbolicity]
\label{thm:stathyp-annulus}
Let $\mu$ be any measure on $\T(S)$ satisfying the thickness property {\rm (P1)} and the separation property 
{\rm (P2)}. Fix a basepoint $x\in \T(S)$ and an arbitrary $k> 0$. Then
\[\lim_{r\to\infty}\frac{1}{r}\frac{1}{\mu(\Ann rxk)^2}\int_{\Ann rxk\times \Ann rxk} d_\T(y,z) \ d\mu(y) d\mu(z)=2.\]
\end{theorem}

We have shown that these hypotheses are satisfied by the standard visual measures $\vis(\vismeas)$ and $\vis(s_x)$, the holonomy measure $\m$, and the Hausdorff measure $\hausmeas$ (Corollaries~\ref{cor:thickness_for_vis}, \ref{cor:thickness_for_hol}, \ref{cor:sep_property}), and thus also all the other measures considered in this paper. Therefore, combining Theorem~\ref{thm:stathyp-annulus} with Lemma~\ref{lem:annuliReduction} (Reduction to annuli) we immediately obtain Theorem~\ref{thm:balls}:

\begin{theorem:balls}
Let $\mu$ denote the Hausdorff measure $\hausmeas$, holonomy measure $\m$, or either standard visual measure $\vis(\vismeas)$ or $\vis(s_x)$.
Then for every point $x\in \T(S)$,
\[\lim_{r\to\infty} \frac 1r \frac{1}{\mu (\Ball rx)^2}\int_{\Ball rx \times \Ball rx} d_{\T}(y,z) \ \ 
d\mu(y) d\mu(z)=2.\]
\end{theorem:balls}

\begin{proof}[Proof of Theorem~\ref{thm:stathyp-annulus}]
Set $\theta= \nicefrac{3}{4}$. For any $0 < \delta ,\sigma < \nicefrac 13$, let $\epsilon = \epsilon(\theta,\sigma) > 0$ be the corresponding thickness parameter guaranteed by Property (P1). For this $\epsilon$ and $\theta'=\nicefrac{1}{2}$, let $C$ and $L$ be the corresponding constants provided by Theorem~\ref{thm:partiallyThickIntervalsAreClose}. 
Properties (P1) and (P2) together imply that for all large $r$, we may restrict to a subset $E_r\subset \Ann rxk\times\Ann rxk$ whose complement has proportional $\mu$--measure at most $\delta$ and such that all pairs $(y,z)\in E_r$ satisfy
\[\thickfrac_{\epsilon}[x,y_t],\thickfrac_{\epsilon}[x,z_t]
\geq \theta \qquad\text{and}\qquad d_\T(y_t,z_t)\geq 3C\]
for all $t\in [\sigma r, r]$, where  $y_t$ and $z_t$ are the time--$t$ points on the geodesic rays from $x$ through $y$ and $z$, respectively. Notice that, in this case, the point $y_t$ cannot be within $C$ of any point on $[x,z]$ (by the triangle inequality, any point on $[x,z]$ within $C$ of $y_t$ must lie in $[z_{t-C},z_{t+C}]$).

We now let $t=2\sigma r$.  Since $\thickfrac_\epsilon[x,y_{2\sigma r}]\geq \theta=\nicefrac{3}{4}$, it follows that the interval $I_r = [y_{\sigma r},y_{2\sigma r}]$ satisfies $\thickfrac_\epsilon I_r\geq \nicefrac{1}{2} = \theta'$ for all large $r$. We also have $\len(I_r) = \sigma r\geq L$ when $r$ is large, and $\sigma<\nicefrac 13$ ensures
$I_r\subset [x,y]$. 
Since 
$I_r \cap \Nbhd_C([x,z])=\emptyset$, 
as noted above, Theorem~\ref{thm:partiallyThickIntervalsAreClose} now implies that  
$I_r \cap \Nbhd_C([y,z])\neq\emptyset.$
 Therefore $[y,z]$ contains a point in the ball $\Ball{2\sigma r+C}{x}$, and so we conclude that
\[d_\T(y,z) \geq 2(r-k - (2\sigma r + C))\]
for all $(y,z)\in E_r$. Putting the above estimates together, we find that
\begin{align*}
\liminf_{r\to\infty}\; &\frac{1}{r}\frac{1}{\mu(\Ann{r}{x}{k})^2}\int_{\Ann{r}{x}{k}\times \Ann{r}{x}{k}}\!d_\T(y,z) \,d\mu(y)d\mu(z)\\
&\geq \liminf_{r\to\infty} \frac{1}{r}(1-\delta)(2r-2k - 4\sigma r -2C)
= (1-\delta)(2-4\sigma).
\end{align*}
Since $\delta$ and $\sigma$ can be chosen arbitrarily small, the result follows.
\end{proof}

\begin{rem}\label{alt-rafi}
One could give an alternate proof of Theorem~\ref{thm:balls} that does not rely on  Theorem~\ref{thm:partiallyThickIntervalsAreClose}, but rather combines Rafi's Theorem~\ref{thm:RafiThinTriangles} with our Theorem~\ref{thm:interval}, which strengthens property (P1) to show that with high probability a sufficiently long geodesic has a totally thick subinterval of definite length. (This holds despite the fact that the probability of such a subinterval occurring at any specified time is small).  
\end{rem}

We similarly obtain Theorem~\ref{thm:spheres}:
\begin{theorem:spheres}
For every point $x\in \T(S)$ and  either family $\{\mu_r\}$ of standard visual measures $\mu_r =\vis_r(\vismeas)$ or $\vis_r(s_x)$ on the spheres $\Sph{r}{x}$, we have
\[E(\T(S),x,d_\T,\{\mu_r\}) = 2.\]
\end{theorem:spheres}
\begin{proof}
Visual measures on $\T(S)$ are constructed by radially integrating these visual measures on spheres. 
In fact, Proposition~\ref{prop:thick-vis} and Theorem~\ref{thm:exp-decay-vis} 
were proved for annuli by first verifying them for spheres, and so analogous formulations of properties (P1) and (P2) also hold for the visual measures $\{\mu_r\}$ on spheres.  The result thus follows by the same argument used to prove Theorem~\ref{thm:stathyp-annulus} above.  \end{proof}

\section{Thin triangles}
For $y,z\in \Ball rx$, we can form the geodesic triangle $\triangle=\triangle(x,y,z)$ whose sides are $[x,y],[x,z]$
and $[y,z]$. Since with positive probability the points $y$ and $z$ are  in the thin part of $\T(S)$, 
we can not expect that almost every triangle is thin as $r\to\infty$. However, it is true that almost every triangle is mostly thin:
For a fixed $\delta$, let  $0\leq \Theta_\delta(\triangle)\leq 1$ denote the proportion of the three sides of $\triangle$ 
consisting of points that lie in a $\delta$--neighborhood of the union of the other two sides. 

\begin{theorem:expthin}
Let $\mu$ denote either the Hausdorff measure $\hausmeas$ or the holonomy measure $\m$. Then for all $x\in\T(S)$ and $\sigma > 0$ there exists  $\delta>0$ such that
$$\liminf_{r\to\infty} \frac{1}{\mu (\Ball rx)^2}\int_{\Ball rx \times \Ball rx} \Theta_\delta(\triangle(x,y,z)) \ d\mu(y)d\mu(z) \geq 1-\sigma.$$
\end{theorem:expthin}

Before embarking on the proof, we use the results in the 
appendix to establish the following analog of Theorem~\ref{thm:thick-hol}.

\begin{lemma}[Thick-stat near the end]
\label{lem:thick_near_end}
For all $0 < \theta,\sigma < 1$, $x\in\T(S)$ there exist $\epsilon''=\epsilon''(\theta,\sigma)>0$ and $\alpha>0$ such that for all $k>0 $ and all sufficiently large $r$ 
 $$\frac{\m \left( \bigl\{y\in \Ann rxk : \thickfrac_{\epsilon''}[y_{(1-\sigma)r},y]< \theta\bigr\}   \right)}{\m(\Ann rxk)}  
 < e^{-\alpha r},$$
where $y_t$ denotes the time--$t$ point along the geodesic ray from $x$ through $y$.
\end{lemma}
\begin{proof}
For the given $\theta$, we let $\rho = 1-\theta$ and take $\delta' = \rho/8$. For this $\delta'$ and the given $\sigma$, we choose $\epsilon_1 \ll \sigma$ and $\tau\geq \tfrac{4\ccc}{\epsilon_1}$ accordingly so that $\kappa\colonequals\frac{\delta'\sigma}h-2\epsilon_1>0$. Recall that every geodesic $[x,y]$ with $y\in \Ann rxk$ then determines a sample path in the set $\P_\tau^x(r(1+\epsilon_1))$ of random walks of at most $r(1+\epsilon_1)/\tau$ steps starting at $x$. As in the Appendix, we let $\mu_\tau^x$ denote the probability measure on $\P_\tau^x(r(1+\epsilon_1))$, and by Lemma~\ref{lem:prob_to_card} we moreover assume that $\tau$ is sufficiently large so that the cardinality of any subset $A\subset \P_\tau^x(r(1+\epsilon_1))$ is bounded by $\mu_\tau^x(A)e^{hr(1+2\epsilon_1)}$.

Fixing $\tau$ and $\rho$ as above and choosing $\epsilon$ so that $x\in \T_\epsilon$, we now apply Theorem~\ref{thm:fraction} with a constant $\gamma$ satisfying $c_\tau^\rho < \gamma < c_\tau^{\rho/2}$. This provides $\epsilon'$ and a cocompact set $\EuScript K \supset \T_\epsilon$ for which the conclusion of that theorem holds. Since $\tau$ is assumed to be large, by Proposition~\ref{EMbounds} we may suppose $c_\tau < e^{-\tau/2}$ so that $\gamma < e^{-\rho\tau/4}$. Roughly, the idea is to now apply Theorem~\ref{thm:fraction} to the set of sample paths starting at $y_{(1-2\sigma)r}$. 

To this end, we first argue that most random walks $\lambda\in \P_\tau^x(r(1+\epsilon_1))$ land in $\EuScript K$ for some step in the interval $[2\sigma  r/\tau,\sigma r/\tau]$ of steps from the end. The subset $\Omega\subset \P_\tau^x(r(1+\epsilon_1))$ of exceptional sample paths that avoid $\EuScript K$ between steps $2\sigma r/\tau$ and $\sigma r/\tau$ from the end is the union over $j\geq 0$ of sets $\Omega_j$ consisting of paths that lie in $\EuScript K$  at step $2\sigma r/\tau+j$ from the end and stay outside $\EuScript K$ for the next $\sigma r/\tau+j$ steps. 
 Now (\ref{eqn:sojournBound}) in the Appendix says that for each $j$ the probability of $\Omega_j$ satisfies 
$$\mu_\tau^x(\Omega_j) \leq M \left(e^{-\tau/4}\right)^{\frac{\sigma r}{\tau} + j} = M e^{-\sigma r / 4} e^{-\tau j / 4}.$$
for some constant $M = M(\tau)$ (note that the constant $\gamma_0$ appearing in \eqref{eqn:sojournBound} satisfies $\gamma_0 < \gamma^{1/\rho} < e^{-\tau/4}$). Assuming that $r$ is large, if we sum over $j$ we see that the probability of a random walk lying in $\Omega$ satisfies
\[\mu_\tau^x(\Omega) \leq e^{-\sigma r/8} < e^{-\delta' \sigma r}.\]

Now consider the paths that do land in $\EuScript K$ for some step in the interval $[2\sigma r/\tau, \sigma r/\tau]$ of steps from the end, namely the set
$$\Sigma:=\P_\tau^x(r(1+\epsilon_1))\setminus \Omega.$$
This can be partitioned into sets $\Sigma_j$ for $j=0,1,\ldots,\sigma r/\tau$,  where a path is in $\Sigma_j$ if its first $\EuScript K$--point in the interval of $[2\sigma r/\tau, \sigma r/\tau]$ steps from the end appears at $n_j=\sigma r/\tau + j$ steps from the end. 
 
We claim that most paths in $\Sigma_j$ will spend at least $\theta$ proportion of their final $n_j$ steps in $\T_{\epsilon'}$. Indeed, by Theorem~\ref{thm:fraction}, the probability that a random walk in $\Sigma_j$ fails to have at least $\theta$ proportion of its final $n_j$ steps $\T_{\epsilon'}$ is bounded by
\[\gamma^{n_j} < e^{-n_j\tau \rho/4} = e^{-2\delta' n_j\tau} = e^{-2(\delta'\sigma r + \delta'\tau j)} < e^{-\delta'\sigma r}.\]
Except for these paths, every path in $\Sigma_j$ spends at least $\theta$ proportion of its last $n_j = \sigma r / \tau + j$ steps in $\T_{\epsilon'}$, and therefore (since $j\leq \sigma r/\tau$) at least $\theta/2$ proportion of its final $2\sigma r/\tau$ steps in $T_{\epsilon'}$.

Since the above bound of $e^{-\delta' \sigma r}$ holds for each set $\Sigma_j$, it follows that the probability that a random walk in $\Sigma$ fails to spend at least $\theta/2$ proportion of its last $2\sigma r/\tau$ steps in $\T_{\epsilon'}$ is at most $e^{-\delta' \sigma r}$. Letting $\Sigma'\subset \Sigma$ denote these exceptional paths, we conclude that
\[\mu_\tau^x(\Sigma') \leq e^{-\delta' \sigma r}\mu_\tau^x(\Sigma) \leq e^{-\delta'\sigma r}.\]
Adding this to our previous estimate on $\mu_\tau^x(\Omega)$, we find that except for the subset $\Omega\cup\Sigma'\subset \P_\tau^x(r(1+\epsilon_1))$ of $\mu_\tau^x$--measure at most $2e^{-\delta' \sigma r}$, every sample path in $\P_\tau^x(r(1+\epsilon_1))$ spends at least $\theta/2$ proportion of its final $2\sigma r/\tau$ steps in $\T_{\epsilon'}$. In particular, by Lemma~\ref{lem:prob_to_card} the number of these exceptional  sample paths is at most
\[\mu_\tau^x(\Omega\cup\Sigma') e^{hr(1+2\epsilon_1)} \leq 2e^{-\delta\sigma r} e^{hr(1+2\epsilon_1)}  = 2e^{hr(1 -\kappa)}.\]
(Recall that $\epsilon_1\ll \sigma$ was chosen so that $\kappa = \frac{\delta'\sigma}h-2\epsilon_1>0$). Since the ball of radius $\ccc$ centered at any point has $\m$--measure $\lmul 1$ by Lemma~\ref{lem:ballVolumes}, it follows that for all large $r$, the set of points $y\in \Ann rxk$ for which the sample path associated to the geodesic $[x,y]$ lies in  $\Omega\cup\Sigma'$ has $\m$--measure at most on the order of 
\[2e^{hr(1-\kappa)} \lmul  \m(\Ann rxk)e^{-\kappa r} \leq \m(\Ann rxk)e^{-\kappa r/2}.\]

We conclude that for all $y\in \Ann rxk$ except for a set of at most this measure, the sample path associated to the geodesic $[x,y]$ spends at least $\theta/2$ proportion of its final $2\sigma r/\tau$ steps in $\T_{\epsilon'}$. Since every point of such a geodesic is within distance $\tau$ of a point on the sample path, by setting $\epsilon'' = \epsilon'e^{-\tau}$ it follows that $\thickfrac_{\epsilon''}[y_{(1-2\sigma)r},y]\geq \theta/2$ for every such $y\in \Ann rxk$. Therefore we have proved the claim for $\theta/2$ and $2\sigma$. As this can be done for any $0<\theta,\sigma < 1$, the theorem follows.
\end{proof}

\begin{proof}[Proof of Theorem~\ref{T:expected_thinness}]
By Corollary~\ref{haus-vs-hol}, it suffices to assume $\mu = \m$, and by Lemma~\ref{lem:annuliReduction} it suffices to estimate the expected value over $\Ann rxk^2$ rather than $\Ball rx^2$. As in the proof of \ref{thm:stathyp-annulus}, we set $\theta = \nicefrac{3}{4}$ and let $\epsilon = \epsilon(\theta,\sigma)$ be the smaller of the constants provided by Property (P1) and Lemma~\ref{lem:thick_near_end}. For this $\epsilon$ and $\theta' = \nicefrac{1}{2}$, we then let $C$ and $L$ be the corresponding constants provided by Theorem~\ref{thm:partiallyThickIntervalsAreClose}.

Properties (P1)--(P2) and Lemma~\ref{lem:thick_near_end} show that for all large $r$ we may restrict to a subset $E_r\subset \Ann rxk^2$ -- with complement having arbitrarily small proportional $\m$--measure -- such that for all $(y,z)\in E_r$ the four intervals
\[
[y_{\sigma r},y_{2\sigma r}],\quad
[z_{\sigma r},z_{2\sigma r}],\quad
[y_{(1-2\sigma)r,}y_{(1-\sigma)r}],\quad\text{and}\quad
[z_{(1-2\sigma)r},z_{(1-\sigma)r}]\]
all have $\thickfrac_\epsilon \geq \theta' = \nicefrac{1}{2}$ and such that $d_\T(y_t,z_t)\geq 3C$ for all $t\geq \sigma r$. By Theorem~\ref{thm:partiallyThickIntervalsAreClose}, it follows that these intervals respectively contain points $y',z',y'',z''$ that each lie within distance $C$ of the geodesic $[y,z]$. Thus by Rafi's fellow traveling result \cite[Theorem C]{Ra}, there is a constant $C'>C$ such that the intervals $[y',y'']$ and $[z',z'']$ are entirely contained within $C'$ of $[y,z]$.

Now, $[y',y'']$ and $[z',z'']$ both have length at least $(1-4\sigma)r$, and so the subintervals of $[y,z]$ that they fellow travel must each have length at least $(1-4\sigma)r - 2C'$. Thus we have identified subintervals of $\triangle(x,y,z)$ whose lengths total at least $4(1-4\sigma)r - 4C'$ and which lie within $C'$ of the union of the other two sides of $\triangle(x,y,z)$. Since the total perimeter of $\triangle(x,y,z)$ is at most $4r$, it follows that
\[\Theta_{C'}(\triangle(x,y,z)) \geq 1-4\sigma - \tfrac{C'}{r}\]
for all $(y,z)\in E_r$. Since the complement of $E_r$ has arbitrarily small proportional $\m$--measure, the result follows.
\end{proof}

\begin{rem}
The above proof may be easily adapted to show that the conclusion of Theorem~\ref{T:expected_thinness} also holds for the expected value of $\Theta_\delta(y_1,y_2,y_3)$ over all triples of points $y_1,y_2,y_3\in \Ball{r}{x}$.
\end{rem}

%%%%
\appendix
\section{Repackaged distance formula}\label{repackaged}

The purpose of this appendix is to repackage Rafi's distance formula (Theorem~\ref{thm:RafiDistFormula}) in a way that treats annular and non-annular subsurfaces on equal footing:

\begin{proposition}[Repackaged distance formula]
Given any sufficiently large threshold $M_0$, for all $x,y\in \T(S)$ we have:
$$
  d_\T(x,y)\ 
\asymp_{M_0}\ 
    d_S(x,y) 
+ \sum_Y \thresh{d_Y(x,y)}_{{M_0}}
$$
Here, the sum is over all (annular and non-annular) proper subsurfaces. 
\end{proposition}

For simplicity and without loss of generality, below we suppose that $\epsilon_0$ has been chosen small enough that $\plog(1/\epsilon_0)\ge 100$, say.
We begin with a straightforward reformulation.
\begin{lemma}\label{lem:reorganizedDistFormula}
For any sufficiently large threshold $M_0$, for all $x,y\in \T(S)$ we have:
\begin{align*}
d_\T(x,y)\  &\asymp_{M_0}\ 
d_S(x,y) 
+ \sum_V \thresh{d_V(x,y)}_{M_0}
+\!\!\! \sum_{  A \,:\, \partial A \in\Gamma_{xy}  }\thresh{d_A(x,y)}_{M_0} \; +\\ 
 &\!\!\sum_{A \,:\, \partial A \notin\Gamma_{xy}}
\!\thresh{\max\left\{\plog\left(d_{\CC(A)}(x,y)\right), \plog\left(\frac{1}{l_x(\partial A)}\right), \plog\left(\frac{1}{l_y(\partial A)}\right)\right\}}_{\log M_0}
\end{align*}
\end{lemma}
\begin{proof}
Since $\Gamma_{xy}$, $\Gamma_x$ and $\Gamma_y$ each contain at most $3g-3$ curves, each $\max$ over these
sets is within bounded multiplicative error of the corresponding sum, and applying a threshold only creates bounded additive error, so the first three terms of the lemma are established.
By the definition of $\Gamma_x$ we have
\[\sum_{\alpha\in \Gamma_x}\plog\left(\frac{1}{l_x(\alpha)}\right) = \sum_{\alpha\notin\Gamma_{xy}} \plog\thresh{\frac{1}{l_x(\alpha)}}_{1/\epsilon_0}.\]
Since this is a sum with at most $3g-3$ nonzero terms, we can increase the threshold to any number $M_0\geq 1/\epsilon_0$
with bounded additive error.  Finally, for functions $f,g,h$, 
we have (in fact with the implied multiplicative constant equal to $3$)
\[
\plog\thresh{f}_{M_0}+\plog\thresh{g}_{M_0}+\plog\thresh{h}_{M_0}\  {\emul}\ 
\thresh{\max\{\plog f,\plog g, \plog h\}}_{\log M_0}.\qedhere
\]
\end{proof}

We now show that each term in the last summand is bilipschitz equivalent to the corresponding hyperbolic distance  $d_A(x,y)$. 

\begin{lemma}\label{lem:twist+lenVShyp}
Consider an annular subsurface $A\subset S$ with core curve $\partial A = \alpha$. For each pair of points $x,y\in \T(S)$, set 
\[H_A(x,y) := \max\left\{\plog\left(d_{\CC(A)}(x,y)\right), \plog\left(\frac{1}{l_x(\alpha)}\right), \plog\left(\frac{1}{l_y(\alpha)}\right)\right\}.\]
If $x,y\in \T(S)$ are such that $\alpha\notin \Gamma_{xy}$ and either $d_A(x,y)$ or $H_A(x,y)$ is greater 
than $36\plog(1/\epsilon_0)$, then  $6^{-1}d_A(x,y) \leq H_A(x,y) \leq 6 d_A(x,y)$.
\end{lemma}

\begin{proof}
Choose points $x,y\in \T(S)$ that satisfy the hypotheses. To fix notation, set $\pi'_\alpha(x) = (0,1)$ and 
$\pi'_\alpha(y) = (d_{\CC(A)}(x,y),1)$. These are the closest-point projections of $\pi_\alpha(x)$ and $\pi_\alpha(y)$ to 
the horocycle bounding 
$\H_\alpha$, and their distances from these points are exactly given by $\plog(1/l_x(\alpha))$ and $\plog(1/l_y(\alpha))$. Let
\[B=d_{\H^2}(\pi_\alpha'(x),\pi_\alpha'(y)) = \text{arccosh}\left(1+\frac{d_{\CC(A)}(x,y)^2}{2}\right)\]
denote the hyperbolic distance between these projections. Using this formula, one may easily check that the inequalities
\begin{equation}\label{eqn:hypLem--goodIneqs}
\plog d_{\CC(A)}(x,y) \leq B \leq 4 \plog d_{\CC(A)}(x,y)
\end{equation}
hold provided that either $B \geq 3$ or $d_{\CC(A)}(x,y)\geq 3$.

Applying the triangle inequality with the points $\pi'_\alpha(x)$ and $\pi'_\alpha(y)$ implies that
\begin{equation}\label{eqn:hypLem-upperbound}
d_A(x,y) \leq \plog\left(\frac{1}{l_x(\alpha)}\right) + B +  \plog\left(\frac{1}{l_y(\alpha)}\right).
\end{equation}
Then \eqref{eqn:hypLem--goodIneqs}, \eqref{eqn:hypLem-upperbound}, and the definition of $H_A$ imply that 
$d_A(x,y)\leq 6H_A(x,y)$ in the case that $B\geq 3$. If $B < 3$, we claim that the hypotheses of the Lemma ensure that $B$ cannot be the largest term on the right-hand side and therefore that $d_A(x,y)\leq 3L \leq 3H_A(x,y)$, where $L$ denotes the larger of the other two terms. Indeed, if $B$ were the largest term and $B<3$, then \eqref{eqn:hypLem-upperbound} 
would imply $d_A(x,y) < 9$, and \eqref{eqn:hypLem--goodIneqs} would necessitate $\plog d_{\CC(A)}(x,y) < 3$  so that 
$H_A(x,y)< 3$.  But then both $d_A$ and $H_A$ are less than $9$, contradicting the hypothesis.

By the above, the assumption $d_A(x,y) \geq 36\plog(1/\epsilon_0)$ implies that $H_A(x,y) \geq 6\plog(1/\epsilon_0)$; 
therefore all cases will be covered by proving that this in turn implies $H_A(x,y)\leq 6d_A(x,y)$.  
Without loss of generality, we may assume that $l_x(\alpha)\leq l_y(\alpha)$; since $\alpha\notin\Gamma_{xy}$ this guarantees $l_y(\alpha)\geq \epsilon_0$. First suppose that $\plog d_{\CC(A)}(x,y) \geq 3 \plog(1/l_x(\alpha))$, in which case we have $\plog d_{\CC(A)}(x,y) = H_A(x,y)\geq 6\plog(1/\epsilon_0)$. In particular we certainly have $d_{\CC(A)}(x,y)\geq 3$; thus 
\eqref{eqn:hypLem--goodIneqs} and the triangle inequality give
\begin{align*}
\plog d_{\CC(A)}(x,y)
&\leq B
\leq \plog\left(\frac{1}{l_x(\alpha)}\right) + d_A(x,y) + \plog\left(\frac{1}{l_y(\alpha)}\right).
\end{align*}
Therefore $H_A(x,y)=\plog d_{\CC(A)}(x,y) \leq 3 d_A(x,y)$ in this case. The remaining possibility $\plog d_{\CC(A)}(x,y)\leq 3\plog(1/l_x(\alpha))$ necessitates $3\plog(1/l_x(\alpha)) \geq H_A(x,y)$. Recall that $\pi'_\alpha(x)$ is the \emph{closest} point projection of $\pi_\alpha(x)$ to the horocycle bounding $\H_\alpha$; since $\pi'_\alpha(y)$ is also on this horocycle we have
\begin{align*}
\plog\left(\frac{1}{l_x(\alpha)}\right) 
&
\leq d_{\H^2}(\pi_\alpha(x),\pi'_\alpha(y))
\leq d_A(x,y) + \plog\left(\frac{1}{l_y(\alpha)}\right).
\end{align*}
The assumptions $3\plog(1/l_x(\alpha)) \geq H_A(x,y)
\ge  6\plog(1/\epsilon_0)$ and $l_y(\alpha)\geq \epsilon_0$ now ensure that $H_A(x,y)\leq 6d_A(x,y)$.
\end{proof}

\begin{corollary}\label{cor:twistVShyp}
Let $H_A(x,y)$ be defined as in Lemma~\ref{lem:twist+lenVShyp}. Then for any threshold ${M_0}\geq 36\plog(1/\epsilon_0)$ 
and any $x,y\in \T(S)$ we have
\[\sum_{\partial A\notin\Gamma_{xy}} 6^{-1}\thresh{d_{A}(x,y)}_{6{M_0}}
\leq \sum_{\partial A\notin\Gamma_{xy}} \thresh{H_A(x,y)}_{{M_0}}
\leq \sum_{\partial A\notin\Gamma_{xy}} 6\thresh{d_{A}(x,y)}_{{M_0}/6}\]
\end{corollary}

With these estimates, we can derive the simplified distance formula.

\begin{proof}[Proof of Repackaged Distance Formula]
Choose any sufficiently large threshold ${M_0}$ such that Lemma~\ref{lem:reorganizedDistFormula} holds for both $e^{6M_0}$ and $M_0/6$ and such that $M_0/6\geq 36\plog(1/\epsilon_0)$. 
Notice that, in any sum of the form $\sum \thresh{f}_{M}$, raising the threshold can only decrease the value of the sum, and lowering the threshold can only increase its value. Therefore, combining Lemma~\ref{lem:reorganizedDistFormula} and Corollary~\ref{cor:twistVShyp} we find that for any $x,y\in \T(S)$ the various distances satisfy
\begin{align*}
d_\T 
&\,\lboth_{\M_0}\, d_S + \sum_V \thresh{d_V}_{e^{6M_0}} + \sum_{\partial A \in \Gamma_{xy}} \thresh{d_A}_{e^{6M_0}} + 
\sum_{\partial A\notin\Gamma_{xy}}\thresh{H_A}_{6M_0}\\
&\leq 6\left(d_S + \sum_V \thresh{d_V}_{{M_0}} + \sum_{\partial A \in \Gamma_{xy}} \thresh{d_{A}}_{{M_0}} + \sum_{\partial A\notin\Gamma_{xy}}\thresh{d_A}_{{M_0}}\right),
\end{align*}
where we have suppressed the $x$ and $y$ in the notation. The lower bound on $d_\T(x,y)$ is similar.
\end{proof}

%%%%
\section{Thickness statistics for random walks}
\label{sec:thickness_for_random_walks}

The purpose of this Appendix is to give a detailed proof of Theorem~\ref{thm:fraction} and its consequence Theorem~\ref{thm:percentage}. A version of the later statement, as well as a sketch of the proof, was indicated in \cite{EM}. However, as our application requires more control of the constants than the statement in \cite{EM} provides, we include a precise formulation. 
 Eskin--Mirzakhani in \cite[\S4.1]{EM} define for each $\tau>0$ a $\Mod(S)$--invariant function  $u_\tau$ which descends to a proper function on the quotient moduli space.  The actual definition of $u_\tau$ is somewhat complicated, but $u_\tau(x)$ is 
 (up to multiplicative constants depending on $\tau$) comparable to  $1/(\text{length of the shortest curve on $x$})$.   Sets of the form   $\{x : u_\tau(x)\leq C\}$ are cocompact subsets of $\T$.   There is a constant $M_\tau$ such that if $d_{\T}(x,y)\leq \tau$ then 
\begin{equation}
\label{eq:boundM}
\frac{u_\tau(x)}{u_\tau(y)}\leq M_\tau.
\end{equation}

Recall from \S\ref{sec:random_walks} that $\N$ denotes a fixed net in $\T$ and that, given a parameter $\tau$,  $\P_\tau^x$ denotes the set of all sample paths  $\lambda\colon[0,1,\dotsc]\to \mathcal N$ (of any length) starting at the basepoint $x$ and satisfying $d_\T(\lambda_i,\lambda_{i+1})\leq \tau$ for all $i$. 
If random sample paths $\lambda=(\lambda_0,\lambda_1,\ldots)$ are constructed via the Markov process in which the net point $\lambda_{i+1}$ is selected uniformly at random among all net points in the ball $\Ball{\tau}{\lambda_{i}}$, then this Markov process determines a probability measure $\mu_\tau^x$ on the set $\P_\tau^x$ of all sample paths (i.e., random walks) starting at $x$. 
Recall that $\P_\tau^x(r)$ denotes the set of sample paths of at most $\floor{r/\tau}$ steps, so that the distance from start to end is at most $r$.
For our applications, we will be concerned about the \emph{number} of sample paths satisfying a certain property, rather than the $\mu_\tau^x$--measure of such paths. These quantities are related in the following manner, which was implicit in \cite{EM}.

\begin{lemma}
\label{lem:prob_to_card}
For any $\delta>0$, there exists $\tau_0$ such that for all $\tau > \tau_0$ and all $r$, the cardinality of any subset $A\subset \P_\tau^x(r)$ satisfies
\[\abs{A}\leq \mu_\tau^x(A)\cdot e^{(h+\delta)r}.\]
\end{lemma}
\begin{proof}
Since $A$ is a finite set, $\mu_\tau^x(A)$ is simply the sum of the measures of the individual elements of $A$. By definition of the Markov process, the $\mu_\tau^x$ measure of an element $\lambda\in A$ is the reciprocal of the product of the number of choices  for the first step of $\lambda$ times the number of choices for the second step, and so on. Thus
\[\mu_\tau^x(\lambda) = \left(\abs{\Ball{\tau}{\lambda_0}\cap \N}\cdot \abs{\Ball{\tau}{\lambda_1}\cap \N}\dotsb \abs{\Ball{\tau}{\lambda_{\floor{r/\tau}-1}}\cap \N}\right)^{-1}\] 
Combining Proposition 4.5 and equation (17) of \cite{EM}, it follows that for the given $\delta$ there exists $\tau_0$ such that for all all $\tau \geq \tau_0$ one has $\abs{\Ball{\tau}{y}\cap \N}\leq e^{(h+\delta)\tau}$ for all $y\in \T(S)$. Therefore, for each $\lambda\in A$ we find that
\[\mu_\tau^x(\lambda) \geq \left( e^{(h+\delta)\tau}\right)^{-\floor{r/\tau}} \geq e^{-(h+\delta)r}.\]
Adding these estimates for all $\lambda\in A$ yields the claimed inequality.
\end{proof}

We now define an averaging operator $A_\tau$
which, given a function $f\colon \N\to \R_+$, produces the new function $A_\tau f\colon \N\to \R_+$ defined as
$$A_\tau f(x) \colonequals \int_{\lambda\in \P_\tau^x}f(\lambda_1)\, d\mu_\tau^x(\lambda).$$
Thus $A_\tau f(x)$ is the expected value of $f(\lambda_1)$ among all random walks starting at $x$.

In \cite{EM},  Eskin and Mirzakhani defined the averaging operator slightly differently. For a function $f\colon \T\to\R_+$ they defined 
$$A_\tau f(x) \colonequals \frac{1}{\m(B_{\T}(x,\tau))}\int_{B_{\T}(x,\tau)} f(y)d\m(y)$$
and then established the following estimates regarding their function $u_\tau$.

\begin{proposition}[Theorem 4.1 of \cite{EM}]
\label{EMbounds}
For all sufficiently large $\tau\geq 0$, there are constants $c_\tau$ and $b_\tau$ such that
\begin{itemize}
\item  $(A_\tau u_\tau)(y) \leq c_\tau u_\tau(y) + b_\tau$ for all $y\in \T$, and 
\item  $c_\tau \leq C' e^{-\tau}$  for a universal constant $C'$ depending only on genus.
\end{itemize}
\end{proposition}

The same estimates hold for our averaging operator over the discrete random walks on the net.
This follows from (17) of \cite{EM} 
which says that the number of net points in a ball of radius $\tau$ is comparable to the volume of the ball, and the fact that on a ball of radius $2\ccc$ the value of $u_\tau$ is (up to multiplicative constants depending on $\ccc$) the value at the center. 

We now state a  general proposition about Markov chains proved by Athreya in \cite{Ath}.  Denote a state space by $\euS$ and let $P_s$ be the probability measure on all random walks in $\euS$ starting at $s$. 
For a subset $C\subset \euS$ and a random walk $X=(X_0,X_1,\dotsc)$ starting at $s=X_0$, denote by $\tau_C(X) \colonequals \inf\{n\geq 0 \mid X_n\in C\}$ the first step at which $X$ enters $C$. 
Proposition 3.1 of \cite{Ath} then gives the following in terms of an averaging operator $A$ defined over steps of the random walk, as above.
\begin{proposition}
\label{prop:sojourns}
Suppose there exist constants $0 < c < 1$ and $b\geq 0$ and a function $V\colon \euS\to \R_+$ defined on the state space of a random walk that satisfies $(AV)(s)\leq cV(s)+b$ for all $s\in \euS$. Then for all $l\geq 0$, all $s\notin C_l\colonequals \{y\in \euS \mid V(y)\leq l\}$, and all $n\geq 0$ we have
\[P_s(\tau_{C_l}(X) > n) \leq \frac{V(s)}{l}\left(c+\frac{b}{l}\right)^n.\]
\end{proposition}

Following Athreya's proof of Theorem 2.3 in \cite[\S6]{Ath}, we now use 
Proposition~\ref{EMbounds} and Proposition~\ref{prop:sojourns}  to get an \emph{exponential} bound on the probability that a random walk $\lambda$ in $\T$ spends a large fraction of its time outside of large compact sets.

\begin{theorem}[Fraction in the thick part; c.f. {\cite[Theorem 2.3]{Ath}}]
\label{thm:fraction}
For all $0 < \rho < 1$,  $\epsilon,\tau>0$, and $\gamma$ such that  $c_\tau^\rho < \gamma <1$, there exists $\epsilon'>0$ and a cocompact set $\EuScript K$ containing $\T_\epsilon$ such that for all $x\in \N\cap \EuScript K$ and all $n\ge 1$ we have
\[\mu_\tau^x\left(\left\{\lambda\in \P_\tau^x :\frac{\abs{\{1\leq k \leq n : \lambda_k\notin \T_{\epsilon'}\}}}{n} > \rho\right\}\right) < \gamma^n.\]
\end{theorem}

As an almost immediate consequence we will have the following (c.f. \cite{EM}).

\begin{theorem}
\label{thm:percentage}
Given $0<\theta<1$, there exist $\delta'>0$ such that for all $\epsilon,\epsilon_1 > 0$ and all large $\tau$ there exists  $\epsilon'$ so that for all $x\in \T_\epsilon\cap\N$ and each $\tau \leq t \leq r$, the estimate
\[\frac{1}{\lfloor t/\tau \rfloor}\left|\bigl\{1\leq i\leq \lfloor t/\tau \rfloor: \lambda_i\in \T_{\epsilon'}\bigr\}\right|\geq \theta\]
holds for all but at most
$$e^{-\delta't }e^{hr(1+2\epsilon_1)}$$
of the sample paths in $\P_\tau^x(r(1+\epsilon_1))$.
\end{theorem}
\begin{proof}
We assume $\tau$ is large enough so that, by Lemma~\ref{lem:prob_to_card}, the cardinality of any subset $A\subset \P_\tau^x(r(1+\epsilon_1))$ satisfies $\abs{A}\leq \mu_\tau^x(A)e^{hr(1+2\epsilon_1)}$. Let $\rho=1-\theta$ and set $\delta'=\rho/4$. For large $\tau$, we have $c_\tau<e^{-\tau/2}$, and we choose $\gamma$ so that  $c_\tau^\rho<\gamma<c_\tau^{\rho/2}<1$.
By Theorem~\ref{thm:fraction} there is a corresponding $\epsilon'$ for these $\epsilon,\tau,\rho,\gamma$ such that if $x\in \T_\epsilon$, then the $\mu_\tau^x$--measure of the set of paths in $\P_\tau^x$ that in their first $n=\lfloor t/\tau\rfloor$ steps spend more than $\rho$ proportion of time outside $T_{\epsilon'}$ is at most $\gamma^n<c_\tau^{n\rho/2}<e^{-n\rho\tau/4}$. 
By Lemma~\ref{lem:prob_to_card} the number of paths that spend more than $\rho$ proportion of time outside the $\T_{\epsilon'}$ (equivalently spending at most $\theta$ proportion inside $\T_{\epsilon'}$) is thus bounded by
\[\gamma^n e^{hr(1+2\epsilon_1)} < e^{-\rho t/4}e^{hr(1+2\epsilon_1)}=e^{-\delta' t}e^{hr(1+2\epsilon_1)}.\qedhere\] 
\end{proof}

We end with the proof of Theorem~\ref{thm:fraction}.

\begin{proof}
 For the given $\gamma$, we have $c_\tau < \gamma^{1/\rho} < 1$. Therefore we may choose $l >0$ sufficiently large so that $\gamma_0 =(c_\tau+\tfrac{b}{l}) < \gamma^{1/\rho}$. We also take $l$ large enough to satisfy $\T_\epsilon \subset C_l = \{y\in \T : u_\tau(y)\leq l\}$ for the given $\epsilon$. Define $\EuScript K\colonequals C_l$, which is cocompact, and fix any point $x\in \N\cap \EuScript K$. Note that $u_\tau(x) \le l$.

For integers $i\geq 0$ define functions $t_i\colon \P^x_\tau \to \mathbb{N}$ as follows:
\begin{itemize}
\item $t_0 = \inf\{k\geq 0  \mid \lambda_k\in C_l\} = 0$,
\item $t_{2i-1} = \inf\{k > t_{2i-2} \mid \lambda_k\notin C_l\}$ for $i \geq 1$, and
\item $t_{2i} = \inf\{k > t_{2i-1} \mid \lambda_k\in C_l\}$ for $i\geq 1$.
\end{itemize}
That is, evaluated on sample path $\lambda\in \P_\tau^x$, $t_0$ is the first time that $\lambda$ is in $C_l$ (which is always the first step since we assume $\lambda_0=x\in C_l$), $t_1$ is the first time that $\lambda$ steps outside of $C_l$, $t_2$ is the next time that $\lambda$ steps into $C_l$, $t_3$ is the next time that $\lambda$ steps outside of $C_l$, and so on. Now define $s_i = t_i - t_{i-1}$ for $i\geq 1$. So:
\begin{itemize}
\item $s_{2i}$ is the number of steps that $\lambda\in \P_\tau^x$ takes \emph{outside} of $C_l$ on its $i^{\text{th}}$ sojourn outside, and
\item $s_{2i-1}$ is the number of steps that $\lambda\in \P_\tau^x$ takes \emph{inside} of $C_l$ on its $i^{\text{th}}$ sojourn inside.
\end{itemize}

Now, for each $\lambda\in \P_\tau^x$, we define a function $F_\lambda\colon\mathbb{N}\to \{0,1\}$ as an indicator for $C_l$, as follows:
\[F_\lambda(k) = \begin{cases}
1, & t_{2i-1}\leq k < t_{2i}\text{ for some $i$};\\
0, & t_{2i} \leq k < t_{2i+1}\text{ for some $i$.}\end{cases}
\]
Then
\[\frac{1}{n}\bigl| \{1 \leq k \leq n : u_\tau(\lambda_k)> l\}\bigr| = \frac{1}{n}\sum_{k=1}^n F_\lambda(k).\]

By definition, we have $\abs{t_{2k} - t_{2k-2}}\geq 2$ for all $k$, which in turn implies that $k \leq t_{2k}$ for all $k\geq 1$. It now follows that
\[\sum_{k=1}^n F_\lambda(k) \leq \sum_{k=1}^{t_{2n}}F_\lambda(k) = \sum_{k=1}^{n}s_{2k},\]
We therefore conclude that
\[\mu_\tau^x\left(\frac{1}{n}\bigl| \{1 \leq k \leq n : u_\tau(\lambda_k)> l\}\bigr| > \rho\right) = \mu_{\tau}^x\left(\sum_{k=1}^n F_\lambda(k) > n\rho\right) \leq \mu_\tau^x\left(\sum_{i=1}^n s_{2i} > n\rho\right).\]

Now, for each $k$, Proposition \ref{prop:sojourns} applied to the function 
$u_\tau$ implies that
\begin{equation}\label{eqn:sojournBound}
\mu_\tau^x(s_{2i} > k) \leq M_\tau\gamma_0^k.
\end{equation}

We now employ a trick to  \emph{exclude short sojourns}: Let $C'\geq 1$ denote some large threshold that will be determined later. For each $i\geq 1$ we now define
\[s_{2i}'=\begin{cases} 0, & s_{2i}\leq C'\\ s_{2i}, & \text{else}\end{cases}\]
 By (\ref{eqn:sojournBound}) above, we have $\mu_\tau^x(s_{2i} = k) \leq \mu_\tau^x(s_{2i} > k-1) \leq M_\tau\gamma_0^{k-1}$ for each $i\geq 1$ and $k\geq 1$. This implies that for all $i\geq 1$ we have
\[\mu_\tau^x(s_{2i}' = k) \leq \begin{cases} 0, & 1\leq k \leq C'\\ 
M_\tau\gamma_0^{k-1}, & k > C'.\end{cases}\]
The point here is that if $s_{2i} \leq C'$ then $\lambda$ takes at most $C'$ steps outside of $C_l$ between $t_{2i-1}$ and $t_{2i}$. By (\ref{eq:boundM}), this implies that $u_\tau(\lambda_k)\leq lM_\tau^{C'}$ for each of these steps (i.e., for $t_{2i-1}\leq k < t_{2i}$). Therefore, by increasing $l$ to $l'\colonequals lM_\tau^{C'}$, the number of steps that $\lambda$ takes outside of $C_{l'}$ (in its first $n$ steps) is bounded above by $\sum_{i=1}^n s'_{2i}$. Therefore, it suffices to find an exponentially small upper bound on $\mu_\tau^x(\sum_{i=1}^n s'_{2i} > \rho n)$.

Our choice of $\gamma_0 < \gamma^{1/\rho}$ ensures that we may choose  $\theta >0$ satisfying the inequalities $\ln(\gamma_0) < -\theta < \ln(\gamma)/\rho$. In particular, we have $e^{\theta}\gamma_0 < 1$. 
Let $g: \P_\tau^x \to \R$ be given by $g=g_\theta=\exp( \theta\cdot\sum\limits_{i=1}^ns'_{2i})$.
Observe that
$$E(g) = \int_{\P_\tau^x}g(\lambda) d\mu_{\tau}^x(\lambda)
\geq \int_{\{g> e^{\theta\rho n}\}} g(\lambda) d\mu_\tau^x(\lambda)
\geq \mu_\tau^x(g> e^{\theta\rho n})\cdot e^{\theta\rho n}.$$
Therefore
\begin{align*}
\mu_\tau^x\left(\sum_{i=1}^n s'_{2i} > \rho n\right)
&= \mu_\tau^x(g> e^{\theta\rho n})
\leq e^{-\theta\rho n}E(g)\\
&= e^{-\theta \rho n}\int_{\P_\tau^x} \left(\prod_{i=1}^n   e^{\theta s'_{2i}(\lambda)}    \right) d\mu_\tau^x(\lambda)
= e^{-\theta\rho n}\prod_{i=1}^n\int_{\P_\tau^x} e^{\theta s'_{2i}(\lambda)} d\mu_\tau^x(\lambda),
\end{align*}
where in the last line we have used the fact that the random variables $s'_1,\dotsc,s'_{2n}$ are independent of each other (so that the expected value of the product is equal to the product of the expected values). We now estimate each expected value in the product. For each $i\geq 1$ we have
\begin{align*}
\int_{\P_\tau^x}e^{\theta s'_{2i}(\lambda)} d\mu_\tau^x(\lambda)
&= \sum_{k=0}^\infty e^{\theta k} \mu_\tau^x(s'_{2i}=k)
= e^{\theta\cdot 0}\mu_\tau^x(s'_{2i}=0) + \sum_{k=\ceil{C'}}^\infty e^{\theta k}\mu_\tau^x(s'_{2i}=k)\\
&\leq 1 + \sum_{k=\ceil{C'}}^\infty e^{\theta k}M_\tau\gamma_0^{k-1}
\leq 1 + \tfrac{M_\tau}{\gamma_0(1-e^{\theta}\gamma_0)}(e^{\theta}\gamma_0)^{C'}.
\end{align*}

Note that the assumptions on $\theta, \gamma,\rho$ imply $e^{\theta\rho}\gamma > 1$. Therefore, since $e^{\theta}\gamma_0 < 1$, we may now choose $C'$ sufficiently large so that
\[\Gamma\colonequals \left(1 + \tfrac{M_\tau}{\gamma_0(1-e^{\theta}\gamma_0)}(e^{\theta}\gamma_0)^{C'}\right) < e^{\theta \rho}\gamma.\]
By the above calculations, we may now conclude that
\begin{align*}
\mu_\tau^x\left(\sum_{i=1}^n s_{2i} > \rho n\right)
&\leq e^{-\theta\rho n}\prod_{i=1}^n\int_{\P_\tau^x} e^{\theta s'_{2i}(\lambda)} d\mu_\tau^x(\lambda)\\
&\leq e^{-\theta\rho n}\left(1 + \tfrac{M}{\gamma_0(1-e^{\theta}\gamma_0)}(e^{\theta}\gamma_0)^{C'}\right)^{n}=\left(e^{-\theta \rho}\Gamma\right)^n 
< \gamma^n.
\end{align*}
This shows that for any $x$ satisfying $u_\tau(x) \leq l$, the probability that a random walk $\lambda\in \P_\tau^x$ spends more than proportion $\rho$  of its time outside of the cocompact set 
$C_{l'} = \{y\in \T : u_\tau(y)\leq lM_\tau^{C'}\}$ during its first $n$ steps is at most $\gamma^n$. Taking $\epsilon'>0$ sufficiently small so that $u_\tau(y)\leq lM_{\tau}^{C'}\implies y\in \T_{\epsilon'}$ completes the proof.
\end{proof}

\end{document}